\documentclass[11pt]{article}
\usepackage{fullpage}
\usepackage{euscript}
\usepackage{amssymb,amsmath,exscale,relsize}
\usepackage{amsfonts}
\usepackage{amsbsy}
\usepackage{amsmath}
\usepackage{epsfig}
\usepackage{amsthm}
\usepackage{amscd}
\usepackage{amstext}
\usepackage{graphicx,dsfont,multirow,ctable,colortbl,rotating,slashbox,float}

\usepackage[pdftex]{hyperref}
\makeatletter
\renewcommand\section{\@startsection {section}{1}{\z@}%
                                   {-3.5ex \@plus -1ex \@minus -.2ex}
                                   {2.3ex \@plus.2ex}%
                                   {\normalfont\large\bfseries}}

\renewcommand\subsection{\@startsection{subsection}{2}{\z@}%
                                     {-3.25ex\@plus -1ex \@minus -.2ex}%
                                     {1.5ex \@plus .2ex}%
                                     {\normalfont\bfseries}}

\newlength{\apb@width}
\newcommand{\autoparbox}[2][c]{\settowidth{\apb@width}{#2}\parbox[#1]{\apb@width}{#2}}

\makeatother

\newcommand{\bea}{\begin{eqnarray}}
\newcommand{\eea}{\end{eqnarray}}
\newcommand{\bee}{\begin{eqnarray*}}
\newcommand{\eee}{\end{eqnarray*}}
\newcommand{\al}{\begin{align*}}
\newcommand{\eal}{\end{align*}}
\newcommand{\be}{\begin{equation}}
\newcommand{\ee}{\end{equation}}
\newcommand{\eq}[1]{(\ref{#1})}
\newcommand{\bem}{\begin{pmatrix}}
\newcommand{\eem}{\end{pmatrix}}

\def\a{\alpha}
\def\b{\beta}
\def\c{\gamma}

\def\e{\epsilon}   
\def\f{\phi}               
  
\def\g{\gamma}
\def\h{\eta}

\def\inf{\infty}

\def\k{\kappa}             
\def\l{\lambda}
\def\m{\mu}
\def\n{\nu}

\def\o{\omega}  
\def\p{\pi}   
       
\def\r{\rho}                                     
\def\s{\sigma}                                   
\def\t{\tau}
\def\th{\theta}
\def\til{\tilde}

\def\D{\Delta}

\def\G{\Gamma}

\def\L{\Lambda}

\def\Tr{{\rm Tr}}

\def\jac{\operatorname{jac}}
\def\reg{\operatorname{reg}}
\def\com{\operatorname{com}}
\def\ex{\operatorname{e}}
\def\tpi{2 \pi i}
\def\SL{\operatorname{SL}}
\def\SL{\operatorname{SL}}
\def\Id{\rm Id}
\def\tr{\operatorname{tr}}
\def\fno{f}

\def\gen{\operatorname{genus}}
\def\QQ{\mathbb{Q}}
\def\PP{\mathbb{P}}

\def \H {{\mathbb H}}

\def \Z {{\mathbb Z}}
\def \C {{\mathbb C}}
\def \R {{\mathbb R}}

\newtheorem{thm}{Theorem}[section]
\newtheorem{cor}[thm]{Corollary}
\newtheorem{lem}[thm]{Lemma}
\newtheorem{prop}[thm]{Proposition}

\theoremstyle{remark}
\newtheorem*{rmk}{Remark}

\begin{document}
\begin{center}


\vspace{1cm} 
{\fontsize{18}{0}\selectfont Ê\bf Ê On Rademacher Sums, the Largest Mathieu Group, \\\vspace{.4cm}and the Holographic Modularity of Moonshine}
\vspace{1.5cm}

Miranda C. N. Cheng$~^{\flat, \natural }$ and John F. R. Duncan$~ ^\sharp$

\vspace{0.5cm}

{\it 
$^\flat$Department of Mathematics, Harvard University,
\\
Cambridge, MA 02138, U.S.A. \\
$^\natural$Department of Physics, Harvard University,
\\
Cambridge, MA 02138, U.S.A. \\
$^\sharp$ Department of Mathematics, Case Western Reserve University,\\
Cleveland, OH 44106, U.S.A.
}

\vspace{1.3cm}

\end{center}

\begin{abstract}

Recently a conjecture has been proposed which attaches (mock) modular forms to the largest Mathieu group. 
This may be compared to monstrous moonshine, in which modular functions are attached to elements of the Monster group. 
One of the most remarkable aspects of monstrous moonshine is the following genus zero property:  the modular functions turn out to be the generators for the function fields of their invariance groups. In particular, these invariance groups define genus zero quotients of the upper half plane. 
It is therefore natural to ask if there is an analogue of this property in the Mathieu case, and at first glance the answer appears to be negative since not all the discrete groups arising there have genus zero. 
On the other hand, in this article we prove that each (mock) modular form appearing in the Mathieu correspondence coincides with the Rademacher sum constructed from its polar part. 
This property, inspired by the AdS/CFT correspondence in physics, was shown previously to be equivalent to the genus zero property of monstrous moonshine.
Hence we conclude that this ``Rademacher summability" property serves as the natural analogue of the genus zero property in the Mathieu case. 
Our result constitutes further evidence that the Rademacher method provides a powerful framework for understanding the modularity of moonshine, and leads to interesting physical questions regarding the gravitational duals of the relevant conformal field theories.
\end{abstract}

\pagebreak
\setcounter{page}{1}
\tableofcontents
\clearpage
\section{Introduction}
\label{Introduction}

In 2010, an unexpected relation between the elliptic genus of a $K3$ surface and the sporadic group $M_{24}$ was observed by Eguchi, Ooguri and Tachikawa \cite{Eguchi2010}.
Subsequently, the connection between the $K3$ elliptic genus twisted by a group element of $M_{24}$ and the characters of $M_{24}$-representations was first explored in \cite{Cheng2010_1,Gaberdiel2010}, and further studied in \cite{Gaberdiel2010a,Eguchi2010a}. From the point of view of the $K3$ superconformal sigma model, these twisted (or twining) elliptic genera are believed to specify the way in which the $K3$ elliptic cohomology furnishes an $M_{24}$-representation of infinite dimension. Notice that such a twisted object depends only on the conjugacy class $[g]$ to which the element $g$ belongs. From the ${\cal N}=4$ superconformal symmetry of the sigma model it follows that they are weak Jacobi forms, and moreover admit natural decompositions leading to a set of (mock) modular forms $H_g$ of weight $1/2$. As a result of these studies the existence of a natural relationship between representations of $M_{24}$ and the mock modular forms $H_g$ has been conjectured. 
While such a relationship is yet to be established, extensive evidence has been gathered in \cite{Cheng2010_1,Gaberdiel2010a,Eguchi2010a}. 
In this context the $H_g$ are often referred to as the McKay--Thompson series and a key prediction of the conjecture is that they exhibit particular modular behaviour under certain discrete subgroups $\G_g$ of  $\SL_2(\R)$.

This may be compared to the moonshine conjecture of Conway and Norton \cite{conway_norton}, which relates genus zero subgroups of $\SL_2(\R)$ to elements of the Monster group. More precisely, to each conjugacy class $[g]$ of the Monster is attached a function $T_g$ (its McKay--Thompson series) whose Fourier coefficients are given by characters of the Monster. The genus zero property of monstrous moonshine then refers to the empirical fact (ultimately proven by Borcherds \cite{borcherds_monstrous}) that $T_g$ defines a function on the upper half plane that generates the function field of its invariance group $\G_g$. In particular, the (compactified) quotient of the upper half plane by $\G_g$ is a Riemann surface of genus zero. This property is also expected to extend to the so-called generalised moonshine \cite{generalized_moonshine} which attaches modular functions to commuting pairs of elements of the Monster. Ogg's observation \cite{ogg} that the normaliser of $\G_0(p)$ in $\SL_2(\R)$ for $p$ a prime is genus zero if and only if $p$ divides the order of the Monster was arguably the first hint of a relation between the Monster and genus zero groups. 

Assuming the validity of the $M_{24}$ conjecture, and in particular, the validity of the descriptions of the $H_g$ given in \cite{Cheng2010_1,Gaberdiel2010a,Eguchi2010a}, a natural question to ask is whether the groups $\G_g \subset SL_2(\R)$ attached to the McKay--Thompson series $H_g$ for $M_{24}$ share this fundamental pattern present in monstrous moonshine: do they have a similar genus zero property? At first glance the answer is negative since not all the groups $\G_g$ arising from $M_{24}$ define genus zero quotients of the upper half plane.
When a pattern breaks, we may seek to replace it with a new pattern which encompasses both the old paradigm and the new ``exception". Moreover, we also want to understand why a pattern exists at all in the first place. These are exactly the goals of this paper, and these considerations lead us to recast the question in a new form: a form suggested by earlier work of I. Frenkel which reconsidered the origin, and interpretation, of the genus zero property of monstrous moonshine.

The article \cite{Duncan2009} makes an important step towards a physical explanation of the genus zero property. There it is shown that a discrete group $\G$ has genus zero if and only if a certain regularisation procedure, inspired by Rademacher's work on the $j$-function \cite{Rad_FuncEqnModInv} and summarised here as ``Reg", renders the Poincar\'e-like series 
\be \label{monster_sum_artistic_impression}
T_\G(\t) = \text{Reg}\Big(\sum_{\g\in\G_\inf\backslash \G}q^{\mathsmaller{-1}}\big\lvert_\g \Big)
\ee
convergent without spoiling invariance under the group $\G$. (We say a group $\G$ has genus zero if it defines a genus zero quotient of the upper half-plane.) 
In the above formula $\G_\inf$ is the subgroup of $\G$ fixing the cusp (representative) $i\inf$, the sum is over a set of representatives $\g$ for the cosets of $\G_{\inf}$ in $\G$, and $q=e^{2\pi i\t}$.
As is turns out, the above expression, which we refer to as a {\em Rademacher sum}, and which can be used to define a function for any discrete subgroup $\G$ of 
${\SL}_2(\R)$ commensurable with $\SL_2(\Z)$ (for example), is anomaly free---that is, invariant under $\G$ even after the regularisation---if and only if the group $\G$ has genus zero. In the case that $\G$ does have genus zero the Rademacher sum \eq{monster_sum_artistic_impression} defines a generator for the field of $\G$-invariant functions, and it defines precisely the generator $T_g$ when $\G=\G_g$ for $g$ in the Monster. So the genus zero property of monstrous moonshine may be reformulated as follows: 
\begin{quote}
The McKay--Thompson series $T_g$  coincides with the Rademacher sum $T_{\G}$ attached to its invariance group $\G=\G_g$. 
\end{quote}
Significantly, Rademacher sums with a higher order pole $q^{-n}$, which are relevant for conformal field theories (CFTs) with higher central charges, also define modular functions as long as the group $\G$ has genus zero. 

A compelling physical interpretation of the Rademacher sum is provided by the so-called {AdS/CFT correspondence} \cite{MaldacenaAdv.Theor.Math.Phys.2:231-2521998} (also referred to as the {gauge/gravity duality} or the {holographic duality} in more general contexts), which asserts, among many other things, that the partition function of a given two dimensional CFT ``with an AdS dual" equals the partition function of another physical theory in three Euclidean dimensions with gravitational interaction and with asymptotically anti de Sitter (AdS) boundary condition and toroidal conformal boundary. 
There is not yet a systematic understanding of the conditions a CFT has to satisfy in order for it to have an AdS dual. Nevertheless, the correspondence, when applicable, provides both deep intuitive insights and powerful computational tools for the study of the theory.
From the fact that the only smooth three-manifold with asymptotically AdS torus boundary condition is a solid torus, it follows that the saddle points of such a partition function are labeled by the different possible ways to ``fill in the torus"; that is, the different choices of primitive cycle on the boundary torus which may become contractible in a solid torus that fills it \cite{MaldacenaJHEP9812:0051998}. 
These different saddle points are therefore labeled by the coset space $\G_\inf\backslash \G$, where $\G=\SL_2(\Z)$ \cite{Dijkgraaf2007}. 
From a bulk, gravitational point of view, the group $\SL_2(\Z)$ has an interpretation as the group of large diffeomorphisms, and $\G_\inf$ is the subgroup that leaves the contractible cycle invariant and therefore can be described by a mere change of coordinates. 
Such considerations underlie the previous use of Rademacher sums in the physics literature \cite{Dijkgraaf2007,Moore2007,BoerJHEP0611:0242006,KrausJHEP0701:0022007,Denef2007,Manschot2007}.

Apart from the partition function, which computes the dimension of the Hilbert space graded by the basic quantum numbers (the energy, for instance) of the theory, more interesting information can be gained in the presence of a non-trivial symmetry group acting on the Hilbert space by studying the twisted partition function (a trace over the Hilbert space with a group element inserted) which computes the group characters of the Hilbert space. 
In the Lagrangian formulation of quantum field theories this twisting corresponds to a modification of the boundary condition. 
Now, in the study of moonshine the McKay--Thompson series attached to the group element $g$ has a physical interpretation as the twisted partition function of a two dimensional CFT with a boundary condition twisted by $g$.
This twisting procedure has a clear geometrical meaning for a two dimensional CFT with an AdS gravity dual: From the point of view of the gravity theory, the change of boundary condition by insertion of a group element $g$ also changes the set of  allowed saddle points, and relatedly, the allowed large diffeomorphisms is now given by a different discrete group $\G_g \subset \SL_2(\R)$. From this argument, assuming that the moonshine CFT (or its cousins with higher central charges) have semi-classical-like AdS descriptions in which the path integrals are controlled by their saddle points, the modular objects should  admit expressions as Rademacher sums over the coset spaces $\G_\inf\backslash \G_g$. This picture motivates the relevance of the {Rademacher summability property}:
\begin{quote}
We say a function $f$ is {\em Rademacher summable} if it coincides with the Rademacher sum attached to its polar part.  
\end{quote}
As we have observed above, Rademacher summability is equivalent to the genus zero property in the monstrous case by the results of \cite{Duncan2009}, where the McKay-Thompson series are modular functions.
Moreover, the Rademacher summing procedure applied to $\G$ produces functions which do not have singularities at cusps of $\G$ other than the infinite one. This is in accordance with our expectation that our theory of gravity should admit a unique low energy description.

The situation is arguably more subtle and also more interesting for $M_{24}$.
As mentioned earlier, in this case the relevant modular objects are weight $1/2$ (mock) modular forms, and one does not expect any clear relation between the genus of $\G_g$ and the Rademacher summability of $H_g$---indeed, not all the $\G_g$ attached to elements of $M_{24}$ have genus zero---while in this work we verify explicitly that the functions $H_g$ are all Rademacher summable:
\be\label{Mathieu_sum_artistic_impression}
H_g(\t) =-2\, \text{Reg} \Big(\sum_{\g\in \G_\inf \backslash \G_g} q^{\mathsmaller{-1/8}} \big\lvert_\g\Big) \;
\ee
for all $g\in M_{24}$.
The precise meaning of the above equation will be elucidated in \S \ref{sec:Radsums}, and is in particular the content of the main theorem ({Theorem \ref{thm:HisR}}) of the present paper. Closely related sums for two of the above (mock) modular forms, corresponding to taking $g$ to be the identity or in a certain class of order 2, have been suggested previously in \cite{Eguchi2009a}. The function $H_g$ for $g$ the identity is one of the mock modular forms constructed over a decade ago by Pribitkin in \cite{Pri_SmlPosWgt_II}. 

Note the remarkable similarity between (\ref{Mathieu_sum_artistic_impression}) and (\ref{monster_sum_artistic_impression}), and recall that Rademacher summability is equivalent to the genus zero property in the monstrous case. In view of this,  the Rademacher summability property clearly serves as a natural analogue of the genus zero property, and as a new paradigm that applies to both the Monster and $M_{24}$. 
Moreover, Rademacher sums provide a new angle from which to view the moonshine phenomenon: is it not miraculous that the Rademacher machinery, applied with such simple inputs, uniformly produces infinite $q$-series that contain so much information about sporadic simple groups? 

The rest of the paper is organised as follows. In \S \ref{M24facts} we review the basic properties of the sporadic group $M_{24}$, focusing on its 24-dimensional permutation representation. In \S \ref{Mock} we review the properties of the (mock) modular forms that have been conjectured to encode an infinite-dimensional, $\Z$-graded $M_{24}$ module. In \S \ref{sec:Radsums}  we explain our construction of the Rademacher sums and state the main theorem (Theorem \ref{thm:HisR}) of the present paper. The rest of the paper is devoted to the proof of Theorem \ref{thm:HisR}.  In \S  \ref{sec:Conv} we discuss the convergence of the Rademacher sums and establish that these sums are indeed convergent given the convergence of certain Selberg--Kloosterman zeta functions, which will be shown in \S \ref{Spec}. Having established the convergence of the sums, in \S \ref{sec:Coeffs}  we give explicit expressions for the Fourier coefficients of our Rademacher sums, and discuss their asymptotic growth. In \S  \ref{sec:Var} we study the transformation properties of the Rademacher sums and establish their (mock) modularity. Given the modularity and information about their behaviour at the cusps, in \S \ref{sec:Coin} we finally establish that the Rademacher sums constructed in \S \ref{sec:Radsums} indeed coincide with the (mock) modular forms which are the proposed McKay--Thompson series for the group $M_{24}$ described in \S \ref{Mock}.

\section{The Largest Mathieu Group} 
\label{M24facts}
\setcounter{equation}{0}

We shall start by recalling some facts about the largest Mathieu group. The group $M_{24}$ may be characterised as the automorphism group of the unique doubly even self-dual binary code of length $24$ with no words of weight $4$, also known as the {\em (extended) binary Golay code}. 
In other words, there is a unique (up to permutation) set, ${\mathcal G}$ say, of length $24$ binary codewords (sequences of $0$'s and $1$'s) such that any other length $24$ codeword has even overlap with all the codewords of ${\mathcal G}$ if and only if this word itself is in ${\mathcal G}$, and the number of $1$'s in each codeword of ${\mathcal G}$ is divisible by and not equal to 4. The group of permutations of the $24$ coordinates that preserves the set ${\mathcal G}$ is the sporadic group $M_{24}$.  See, for instance, \cite{sphere_packing}. 
 \begin{table}[h] \centering  \begin{tabular}{ccccccc}
 \toprule
$[g]$ & cycle shape & $\eta_g(\t)$ & $k_g$ &$n_g$& $N_g$ &$h_g$ \\\midrule
$1A$ & $1^{24}$ & $ \eta^{24}(\t) $ & 12 &1  &1 &1\\
$2A$ & $1^{8}2^8$ & $\eta^{8}(\t)\eta^{8}(2\t)$ & 8&2&2 &1\\
$2B $& $2^{12}$ & $\eta^{12}(2\t)$ & 6 &2& 4&2\\
$3A$&$1^63^6$&$\h^6(\t)\h^6(3\t)$&6&3&3&1\\
$3B$&$3^8$&$\h^8(3\t)$&4&3&9&3\\
$4A$& $2^4 4^4$ & $\h^4(2\t)\h^4(4\t)$ &4&4&8&2\\
$4B$&$1^4 2^4 4^4$&$\h^4(\t)\h^2(2\t)\h^4(4\t)$&5&4&4&1\\
$4C$&$4^6$&$\h^6(4\t)$&3&4&16&4\\
$5A$&$1^45^4$&$\h^4(\t)\h^4(5\t)$&4&5&5&1\\
$6A$&$1^22^23^26^2$&$\h^2(\t)\h^2(2\t)\h^2(3\t)\h^2(6\t)$&4&6&6&1\\
$6B$&$6^4$&$\h^4(6\t)$&2&6&36&6\\
$7AB$&$1^3 7^3$&$\h^3(\t)\h^3(7\t)$&3&7&7&1\\
$8A$&$1^2 2^14^18^2$&$\h^2(\t)\h(2\t)\h(4\t)\h^2(8\t)$&3&8&8&1\\
$10A$&$2^210^2$&$\h^2(2\t)\h^2(10\t)$&2&10&20&2\\
$11A$&$1^2 11^2$&$\h^2(\t)\h^2(11\t)$&2&11&11&1\\
$12A$&$2^14^16^112^1$&$\h(2\t)\h(4\t)\h(6\t)\h(12\t)$&2&12&24&2\\
$12B$&$12^2$&$\h^2(12\t)$&1&12&144&12\\
$14AB$&$1^1 2^1 7^114^1$&$\h(\t)\h(2\t)\h(7\t)\h(14\t)$&2&14&14&1\\
$15AB$&$1^1 3^1 5^115^1$&$\h(\t)\h(3\t)\h(5\t)\h(15\t)$&2&15&15&1\\
$21AB$&$3^1 21^1$&$\h(3\t)\h(21\t)$&1&21&63&3\\
$23AB$&$1^1 23^1$&$\h(\t)\h(23\t)$&1&23&23&1\\
 \bottomrule
  \end{tabular}
   \caption{\label{examples_eta} \footnotesize{The cycle shapes, weights $(k_g)$, levels $(N_g)$ and orders $(n_g)$ of the 26 conjugacy classes of the sporadic group $M_{24}$. The length of the shortest cycle is $h_g=N_g/n_g$. The naming of the conjugacy classes follows  the ATLAS convention (cf. \cite{atlas}). We write $7AB$, for example, to indicate that the entries of the incident row are valid for both the conjugacy classes $7A$ and $7B$.
}}
  \end{table}

As such, $M_{24}$ naturally admits a permutation representation of degree $24$, and this allows us to assign a {\em cycle shape} to each of its elements. For example, to the identity element we associate the cycle shape $1^{24}$; to an element of $M_{24}$ that is a product of $12$ mutually commuting transpositions we associate the cycle shape $2^{12}$, and so on. More generally, any cycle shape arising from an element of $M_{24}$ (or $S_{24}$, for that matter) is of the form
$$
{i_1}^{\ell_1} {i_2}^{\ell_2} \dotsi{i_r}^{\ell_r} ,\quad \sum_{s=1}^{r} \ell_{s} \,i_{s} =24\;,
$$
for some $\ell_s\in\Z^+$ and $1\leq i_1<\cdots<i_r\leq23$ with $r\geq 1$. 
Clearly, the cycle shape of an element of $M_{24}$ depends only on its conjugacy class, denoted by $[g]$, although different conjugacy classes can share the same cycle shape.
For future reference we denote the character underlying this {defining} $24$-dimensional representation of $M_{24}$ by $\chi$. 
The value of the character $\chi(g)$ equals the number of fixed points of $g$ in the action on the set of 24 points. In particular, note that $\chi(g)=\ell_1$ in case $i_1=1$ and $\chi(g)=0$ otherwise.

It turns out that the cycle shapes of $M_{24}$ have many special properties that will be important for the understanding of the modular properties of the associated McKay--Thompson series which we will discuss shortly. 
First, the $M_{24}$ cycle shapes are privileged in that they are all of the so-called {\em balanced type} (cf. \cite{conway_norton}), meaning that for each $g\in M_{24}$ there exists a positive integer $N_g$ such 
that if $\prod i_s^{\ell_s}$ is the cycle shape associated to $g$ then
$$
	\prod_si_s^{\ell_s}=\prod_s\left(\frac{N_g}{i_s}\right)^{\ell_s}\;.
$$
We will refer to the number $N_g$ as the {\em level} of the $g$. 

If $g$ has cycle shape $i_1^{\ell_1}\cdots i_r^{\ell_r}$ then the order of $g$ is the least common multiple of the $i_s$'s. 
A second special property of $[g]\subset M_{24}$ is that the order of $g$ coincides with the length $i_r$ of the longest cycle in the cycle shape. 
Henceforth we will denote $n_g=i_r$. 

Finally, observe that for all $g\in M_{24}$ the level $N_g$ defined above equals the product of the shortest and the longest cycle. 
Hence we have $h_g n_g = N_g$ where $h_g$ denotes the length of the shortest cycle in the cycle shape. 
Moreover, we also have the property $h_g \lvert n_g$ and $h_g\lvert 12$. This is very reminiscent of the monstrous moonshine \cite{conway_norton}. 
We also set $k_g=\sum_{s=1}^r \ell_s$/2 to be half of the total number of cycles and call it the {\em weight} of $g$. 
Of course, $N_g$, $n_g$, $h_g$ and $k_g$ depend only on the conjugacy class $[g]$ containing $g$ and can be found in Table \ref{examples_eta}.

To each element $g\in M_{24}$ we can attach an {\em eta-product}, to be denoted $\eta_g$, which is the function on the upper half-plane given by
\begin{gather}
	\eta_g(\tau)=\prod_s\eta(i_s\tau)^{\ell_s}
\end{gather}
where $\prod_si_s^{\ell_s}$ is the cycle shape attached to $g$, and $\eta(\tau)$ is the Dedekind eta function satisfying $\eta(\tau)=q^{1/24}\prod_{n\in\Z^+}(1-q^n)$ for $q=e(\tau)$, where for later convenience, here and everywhere else in this paper we will use the shorthand notation 
$$e(x)=e^{2\p i x}\;.$$ 
As was observed in \cite{Mason,DummitKisilevskyMcKay}, the eta-product $\eta_g$ associated to an element $g\in M_{24}$ (or rather, to its conjugacy class $[g]$) is a cusp form of  weight $k_g$ for the group $\G_0(N_g)$, with a {\em Dirichlet character} $\varsigma_g$ that is trivial if the weight $k_g$ is even and is otherwise defined, in terms of the Jacobi symbol $(\frac{n}{m})$, by
$$
\varsigma_g(\g)=
\begin{cases}
	\left(\frac{N_g}{d}\right) (-1)^{\frac{d-1}{2}},& 
	\text{$d$ odd,} \\
	\left(\frac{N_g}{d}\right), & 
	\text{$d$ even,} 
\end{cases}
$$
in case $d$ is the lower right entry of $\g \in \G_0(N_g)$. 
Let's recall that 
$$
\G_0(N)=\left\{\g \big\lvert\; \g = \bem a&b \\ c&d \eem \in \SL_2(\Z)\,,\; c= 0 \;\;{\rm mod}\;\; N \right\}\;.
$$

The eta-product $\eta_g$ also defines a cusp form of weight $k_g$ on the larger (or equal) group $\G_0(n_g) \supseteq \G_0(N_g)$ if we allow for a slightly more sophisticated multiplier system. 
We remark here that, according to our conventions, a function $\xi:\G\to \C^*$ is called a {\em multiplier system for $\G$ of weight $w$} in case the identity
\be\label{defn:multsys}
	\xi(\g\sigma)
	\jac(\g\sigma,\t)^{w/2}
	=
	\xi(\g)
	\xi(\sigma)
	\jac(\g,\sigma\t)^{w/2}
	\jac(\sigma,\t)^{w/2}
\ee
holds for all $\g,\sigma\in\G$ and $\t\in\H$. Writing $\g= \big(\begin{smallmatrix}a&b\\c&d\end{smallmatrix}\big)$, we have in the above formula $\g\t=\frac{a\t+b}{c\t+d}$, and the Jacobian $$\jac(\g,\t)=(c\t+d)^{-2}\;.$$
In detail, for $\g\in \G$, we define the {\it slash operator $\lvert_{\xi,w}$ of weight $w$ associated with multiplier $\xi$} as 
\be\label{slash}
(f|_{\xi,w}\g)(\t)
	=
	\xi(\g)f(\g\t)\jac(\g,\t)^{\mathsmaller{w/2}}\;,
\ee
and in particular a holomorphic function $f: \H\to \C$ is a modular form of weight $w$ and multiplier $\xi$ on $\G$ if and only if $(f|_{\xi,w}\g)(\t)=f(\t)$ for all $\g\in \G$. 
In the present case we have 
\be\label{eqn:G0nvaretag}
\Big(\frac{1}{\h_g}\Big\lvert_{\xi_g,k_g} \g\Big)(\t) = \frac{1}{\h_g(\t)}\quad,\quad \text{for all}\quad \g \in \G_0(n_g) \;,
\ee
where $\xi_g(\g) = \r_{n_g|h_g}(\g) \varsigma_g(\g)$ with 
\be\label{rho_multiplier}
\rho_{n|h}(\g)=\ex(-\tfrac{1}{(\g\infty-\g 0)nh})=\ex(-cd/nh)\;.
\ee
Note that $\rho_{n_g|h_g}$ is actually a character on $\G_0(n_g)$ since we have that $xy\equiv 1 \pmod{h_g}$ implies $x\equiv y\pmod{h_g}$ by virtue of the fact that $h_g=N_g/n_g$ is a divisor of $24$ for every $g\in M_{24}$ (cf. \cite[\S3]{conway_norton}). Evidently the kernel of $\rho_{n|h}$ is $\G_0(nh)$, and in particular, $\rho_{n|h}$ is trivial on $\G_{\inf}$.

These (meromorphic) modular forms $1/\h_g(\t)$ have also the interpretation as the partition function of a the conformal field theory of 24 free chiral bosons, twisted by $g\in M_{24}$ which acts on the 24  bosons as an element of $S_{24}$. As such, they are also McKay--Thompson series whose Fourier coefficients are positive-integral linear combinations of $M_{24}$ characters. See Table \ref{eta_prod_decomp} for the first ten $M_{24}$-representations appearing in $1/\h_g(\t)$. 
Following an earlier observation in \cite{Govindarajan2009}, it was shown in \cite{Cheng2010_1} that they are connected to the elliptic genera of $K3$ surfaces and the moonshine for $M_{24}$ which will  be discussed in the next section, via a lifting to Siegel modular forms and generalised Kac-Moody superalgebras. More details will be given in \cite{to_appear}.

Observe that the behaviour of the multipliers naturally divides the conjugacy classes of $M_{24}$ into two types: those for which the phase is trivial on $\G_0(n_g)$ and those for which it is not. The former occurs just when $h_g=1$, and the elements of $g\in M_{24}$ with $h_g>1$ are exactly the classes which act fixed-point-freely in the defining permutation representation on $24$ points. The elements with $h_g=1$ have at least one fixed point and thus can be located in a (maximal) subgroup of $M_{24}$ isomorphic to the second largest Mathieu group, $M_{23}$.

\section{Mock Modular Forms} 
\label{Mock}
\setcounter{equation}{0}

In the recent article \cite{Eguchi2010} a remarkable observation relating $M_{24}$ and the unique (up to scale) weak Jacobi form of weight zero and index one, here denoted $Z(\t,z)$, was made via a decomposition of the latter object into a combination of mock theta series and mock modular forms. As is shown in \cite{Eguchi1989,Eguchi2008,Eguchi2009a}, the function $Z(\t,z)=8\sum_{i=2,3,4} (\frac{\th_i(\t,z)}{\th_i(\t,0)})^2$ admits an expression
\be\label{expand_mu}
Z(\t,z) = \frac{\th_1(\t,z)^2}{\eta(\t)^3}\left(a \,\m(\t,z) +q^{-1/8}\big(b  + \sum_{k=1}^\inf t_k\, q^k \big) \right)
\ee
with $a, b,t_k\in\Z$ for all $k\in \Z^+$ where $\m(\t,z)$ denotes the {\em Appell-Lerch sum}, satisfying
$$
 \m(\t,z) = \frac{-i y^{1/2}}{\th_{1}(\t,z)}\,\sum_{\ell=-\inf}^\inf \frac{(-1)^{\ell} y^n q^{\ell(\ell+1)/2}}{1-y q^\ell}
$$
for $q=e(\t)$ and $y=e(z)$. 
The Jacobi theta functions are given by
\begin{align}\notag
\th_1(\t,z) &= -i q^{1/8} y^{1/2} \prod_{n=1}^\inf (1-q^n) (1-y q^n) (1-y^{-1} q^{n-1})\\ \notag
\th_2(\t,z) &=  q^{1/8} y^{1/2} \prod_{n=1}^\inf (1-q^n) (1+y q^n) (1+y^{-1} q^{n-1})\\ \notag
\th_3(\t,z) &=  \prod_{n=1}^\inf (1-q^n) (1+y \,q^{n-1/2}) (1+y^{-1} q^{n-1/2})\\
\th_4(\t,z) &=  \prod_{n=1}^\inf (1-q^n) (1-y \,q^{n-1/2}) (1-y^{-1} q^{n-1/2})\;.
\end{align}
 By inspection, $a=24=\chi(1A)$, $b=-2$ and the first few $t_k$ are
$$
2\times45,\, 2\times231,\, 2\times770,\, 2\times2277, \,2\times5796\cdots
$$
These positive integers $t_k$ have the interpretation of enumerating the $k$-th massive representations of the ${\cal N}=4$ superconformal algebra in the elliptic genus of a $K3$ surface.

The surprising connection to $M_{24}$, beyond the fact that $a=24$ is the dimension of the defining permutation representation of $M_{24}$, is the following: the integers $45$, $231$, $770$, $2277$ and 5796, which are the $t_k/2$ for $k=1,2,3,4,5$, are the dimensions of irreducible representations of $M_{24}$. It was conjectured that the entire set of values $t_{k}$ for $k\in \Z^+$ encode the graded dimension of a naturally defined graded $M_{24}$ module $K=\bigoplus_k K_k$ with ${\rm dim}\,K_k=t_k$. If the conjecture holds then we can obtain new functions by replacing $t_k$ with ${\rm tr}_{K_k}g$ for $g\in M_{24}$. This idea was first investigated in \cite{Cheng2010_1}, and independently in \cite{Gaberdiel2010}; see also \cite{Gaberdiel2010a} and \cite{Eguchi2010a}. It is an idea that we pursue further in this article.

Define $H(\tau)$ so that $Z(\t,z)\eta(\t)^3=\th_1(\t,z)^2(a\mu(\t,z)+H(\t))$. Then
\be\label{H_qexp}
H(\t) =q^{-\frac{1}{8}}\left(-2  + \sum_{k=1}^\inf t_kq^k\right)
\ee
so that $H(\tau)$ is (essentially) a power series incorporating the $t_k$, and hence the graded dimension of the conjectural $M_{24}$ module $K$. 
This function $H(\tau)$ enjoys a special relationship with the group $\SL_2(\Z)$; namely, it is a {\em weakly holomorphic mock modular form of weight $1/2$} on $\SL_2(\Z)$ with {\em shadow} $24 \,\eta(\tau)^3$ (cf. \cite{Atish_Sameer_Don}), which means that $H(\t)$ is a holomorphic function on the upper half-plane $\H$ with at most exponential growth as $\t\to\alpha$ for any $\alpha\in \mathbb{Q}$, and if we define the {\em completion} of $H(\tau)$, to be denoted $\hat{H}(\t)$, by setting 
$$
\hat{H}(\t)=H(\t)+24\, (4{i})^{-1/2} \int_{-\bar \t}^{\infty}(z+\t)^{-1/2}\overline{\eta(-\bar z)^3}{\rm d}z,
$$ 
then $\hat{H}(\t)$ transforms as a modular form of weight $1/2$ on $\SL_2(\Z)$ with multiplier system conjugate to that of $\eta(\tau)^3$. In other words, we have
$$
\big(\hat{H}(\t)\lvert_{\e^{\mathsmaller{-3}},\mathsmaller{1/2}}\mathlarger{\g}\big)(\t)= \epsilon(\gamma)^{\mathsmaller{-3}}\hat{H}(\g\t)
\jac(\g,\t)^{\mathsmaller{1/4}}=\hat{H}(\t)
$$
for $\gamma\in \SL_2(\Z)$, where $\epsilon: \SL_2(\Z)\to \C^*$ is the multiplier system for $\eta(\t)$ satisfying 
$$\big(\h\lvert_{\e,\mathsmaller{1/2}}\mathlarger{\g}\big)(\t)=\eta(\t)\;.$$
(See \S\ref{Dedeta} for an explicit description of $\e$.) 

More generally, a holomorphic function $h(\t)$ on $\H$ is called a {\em (weakly holomorphic) mock modular form of weight $w$} for a discrete group $\G$ (e.g. a congruence subgroup of $\SL_2(\R)$) if it has at most exponential growth as $\t\to\alpha$ for any $\alpha\in \mathbb{Q}$, and if there exists a holomorphic modular form $g(\t)$ of weight $2-w$ on $\G$ such that $\hat{h}(\t)$, given by
\be\label{def_mock}
\hat{h}(\t)=h(\t)+\left(4{i}\right)^{w-1}\int_{-\bar \t}^{\infty}(z+\t)^{-w}\overline{g(-\bar z)}{\rm d}z,
\ee
is a (non-holomorphic) modular form of weight $w$ for $\G$ for some multiplier system $\psi$ say. In this case the function $g$ is called the {\em shadow} of the mock modular form $h$ and we call $\psi$ the multiplier system of $h$.

\begin{table}[h!] \centering  \begin{tabular}{ccc}
 \toprule
$[g]$ & $\chi(g)$ & $\tilde{T}_g(\tau)$\\\midrule
$1A$ & $24$ & $0 $\\
$2A$ & ${8}$ & $16\Lambda_2$\\
$2B $&0& $-24\Lambda_2+8\Lambda_4=2\eta(\tau)^8/\eta(2\tau)^4$\\
$3A$&6& $6\Lambda_3$\\
$3B$&0& $2\eta(\tau)^6/\eta(3\tau)^2$\\
$4A$&0& $4\Lambda_2-6\Lambda_4+2\Lambda_8=2\eta(2\tau)^8/\eta(4\tau)^4$\\
$4B$&4& $4(-\Lambda_2+\Lambda_4)$\\
$4C$&0&$ 2\eta(\tau)^4\eta(2\tau)^2/\eta(4\tau)^2$\\
$5A$ &4&$ 2\Lambda_5$\\
$6A$&2&$ 2(-\Lambda_2-\Lambda_3+\Lambda_6)$\\
$6B$&0&$ 2\eta(\tau)^2\eta(2\tau)^2\eta(3\tau)^2/\eta(6\tau)^2$\\
$7AB$&3&$ \Lambda_7$\\
$8A$&2&$ -\Lambda_4+\Lambda_8$\\
$10A$& 0&$ 2\eta(\tau)^3\eta(2\tau)\eta(5\tau)/\eta(10\tau)$\\
$11A$& 2&$ 2(\Lambda_{11}-11\eta(\tau)^2\eta(11\tau)^2)/5$\\
$12A$& 0&$ 2\eta(\tau)^3\eta(4\tau)^2\eta(6\tau)^3/\eta(2\tau)\eta(3\tau)\eta(12\tau)^2$\\
$12B$& 0&$ 2\eta(\tau)^4\eta(4\tau)\eta(6\tau)/\eta(2\tau)\eta(12\tau)$\\
$14AB$& 1&$ (-\Lambda_2-\Lambda_7+\Lambda_{14}-14\eta(\tau)\eta(2\tau)\eta(7\tau)\eta(14\tau))/3$\\
$15AB$& 1&$ (-\Lambda_3-\Lambda_5+\Lambda_{15}-15\eta(\tau)\eta(3\tau)\eta(5\tau)\eta(15\tau))/4$\\
$21AB$& 0&$ (7\eta(\tau)^3\eta(7\tau)^3/\eta(3\tau)\eta(21\tau)-\eta(\tau)^6/\eta(3\tau)^2)/3$\\
$23AB$& 1&$ (\Lambda_{23}-23\f_{23,1}+23\f_{23,2})/11$\\
 \bottomrule
  \end{tabular}
   \caption{\label{h_g} \footnotesize{In this table we collect the data that via equation \eq{h_g_explicit} define the weight $1/2$ (mock) modular forms $H_g(\t)$. For $N \in \Z_+$, we denote by $\L_N$ the weight 2 modular form on $\G_0(N)$ given by $\L_N = N q\frac{d}{dq} (\log \frac{\h(N\t)}{\h(\t)})$ (cf. \eq{eqn:psin}). For $N=23$, there are two new forms and we choose the basis $\f_{23,1} = \h^2_{23\!A\!B}$ and $\f_{23,2}$ given in \eq{phi232}.
}}
  \end{table}

If the conjectural $M_{24}$-module $K$ exists, apart from $H(\t)$ there must be a family of functions $H_g(\t)$, for each conjugacy class $[g] \subset M_{24}$, obtainable by replacing each $t_k$ with  
the trace of $g$ on $K_k$. These are the McKay--Thompson series given by
\be\label{Hg_qexp}
H_g(\t) =q^{-\frac{1}{8}}\left(-2  + \sum_{k=1}^\inf {\rm tr}_{K_k}(g)q^k\right)\;.
\ee
Explicit expressions for these function $H_g(\t)$ have been proposed in \cite{Cheng2010_1,Gaberdiel2010}, mostly for conjugacy classes with $h_g=1$, and completed for all $[g]\subset M_{24}$ in \cite{Gaberdiel2010a,Eguchi2010a}. These proposals state that they are given by $H(\t)$, the character $\chi(g)$ of $g$ in the 24-dimensional representation, and certain weight two modular forms $\til T_g(\t)$ for the group $\G_0(N_g)$, by 
\be\label{h_g_explicit}
H_g(\t) = \frac{\chi(g)}{24} H(\t) - \frac{\til T_g(\t)}{\h(\t)^3}\;.
\ee
These data are collected in Table \ref{h_g}. The expression \eq{h_g_explicit} makes manifest that the $q$-series $H_g(\t)$ is a mock modular form with shadow ${\chi(g)} \h^3(\t)$, and is in particular a usual modular form when $h_g > 1$, or equivalently, $\chi(g)=0$. Moreover, it is easy to check that $H_g(\t)$  transforms nicely under the group $\G_0(n_g)$, with the multiplier system  $\psi(\g)=\e(\g)^{\mathsmaller{-3}}\r_{n_g|h_g}(\g) $. Notice that the extra multiplier $\r$ that appears when $h_g\neq 1$ is the same as that of the inverse eta-products $1/\h_g(\t)$ which are also related to $M_{24}$, a fact that is in accordance with the $1/2$- and $1/4$-BPS spectrum of the ${\cal N}=4$, d=4 theory obtained by $K3\times T^2$ compactification of the type II string theory. These mock modular forms $H_g(\t)$ are the central objects of study in this work.

Although the conjecture \eq{Hg_qexp} stating that the Fourier coefficients of $H_g(\t)$ are all given by characters of $M_{24}$-module  still remains to be proven, it has been checked up to the first 1000 terms. We list the first few in Table \ref{Mock_decompositions}.
Together with other pieces of evidence, this makes it an extremely plausible conjecture. 
In this paper we will hence assume the validity of this conjecture as a part of our motivation, while our main result (Theorem \ref{thm:HisR}) holds independent of it.

\section{Rademacher Sums} 
\label{sec:Radsums}
\setcounter{equation}{0}
In this section we will state and explain the main theorem (Theorem \ref{thm:HisR}) of the paper, while postponing the proof till the later sections.

Recall that $H(\t)$ is a mock modular form of weight $1/2$ on $\SL_2(\Z)$ with a multiplier system given by $\e(\g)^{\mathsmaller{-3}}$ and with leading term $-2q^{\mathsmaller{-1/8}}$ in its $q$-expansion. We may consider the problem of attaching a function with these data via Rademacher sums, and also analogous functions to discrete subgroups of $\SL_2(\R)$ other than the modular group\footnote{For the purpose of establishing relations to $M_{24}$ it suffices to look at subgroups of $\SL_2(\Z)$, while it is possible to generalise the analysis to $\SL_2(\R)$ subgroups commensurable to $\SL_2(\Z)$}. 

Let $\G$ be a finite index subgroup of $\SL_2(\Z)$ containing the group $\G_{\infty}$ of upper triangular matrices in $\SL_2(\Z)$. In particular, we assume that $\G$ contains $-\Id$. Define $R_{\G}(\t)$ by setting
\be\label{defn_R_Gamma}
R_{\G}(\t)=\lim_{K\to \infty}\sum_{\g\in(\G_{\infty}\backslash\G)_{<K}}
	\e(\g)^{\mathsmaller{-3}}
	\ex(-\tfrac{\g\t}{8})
	\reg(\g,\t)
	\jac(\g,\t)^{\mathsmaller{1/4}}\;,
\ee
where $\reg(\g,\t)$ is the {\em Rademacher regularisation factor (of weight $1/2$ and index $-1/8$)}, defined by $\reg(\g,\t)=1$ in case $\g$ is upper triangular (i.e. $\g\cdot\infty=\infty$) and
\be\label{defn:reg}
\reg(\g,\t)=\ex(\tfrac{\g\t-\g\infty}{8})\,\ex(\tfrac{\g\infty-\g\t}{8},\tfrac{1}{2})
\ee
otherwise, where $\ex(x,s)$ is the following generalisation of the exponential function:
\be\label{defn:genexp}
\ex(x,s)=\sum_{m\geq 0}\frac{(\tpi x)^{m+s}}{\G(m+s+1)}\;.
\ee
Note that we recover $\ex(x)=e^{\tpi x}$ by taking $s=0$ in $\ex(x,s)$, and for $n$ a positive integer $\ex(x,n)$ is the difference between $\ex(x)$ and the order $n-1$ Taylor approximation to $\ex(x)$. It is also closely related to the incomplete Gamma function $\G(s,x)=\int_x^{\infty}t^{s-1}e^t{\rm d}t$: applying  integration by parts 
repeatedly, we obtain
\be\label{defn:genexpintegral}
\ex(x,s)= \frac{\ex(x)}{\G(s)} \int^{2\p i X}_0 t^{s-1}e^t{\rm d}t \;.
\ee

Let $T$ be the element of $\G_{\infty}$ such that $\tr(T)=2$ and $T\t=\t+1$.  Note that two elements $\g,\sigma$ of $\G$  are in the same right coset of $\G_{\infty}$ if and only if there is some $n\in \Z$ such that $\sigma = T^n \g$ or $-T^n\g$. In particular, they have the same lower rows up to multiplication by $\pm\Id$ and we can therefore use  the lower rows to parametrise the cosets. 
This motivates the sum over the ``rectangle" in (\ref{defn_R_Gamma}): 
For any given positive $K$, we consider a sum over a set $\{\g\}$ of representatives for the cosets of $\G_{\infty}$ in $\G$ whose elements have lower rows $(c,d)$ satisfying $0\leq c<K$ and $-K^2<d<K^2$. Using a superscript ${}^\times$ to indicate a sum over coset representatives other than the trivial one, we have
\be\label{defn:GinfGKcross}
	(\G_{\inf}\backslash\G)_{<K}^{\times}=
	\left\{\G_{\inf}\g\mid 0<c<K,\;-K^2<d<K^2
	\right\}\;.
\ee
Here and everywhere else we use $(c,d)$, a shorthand for $(c(\g),d(\g))$, to denote the lower row entries when writing $\g$ as a $2\times 2$ matrix. 
Complimented with a term given by the trivial coset (the one with a representative $\g=$Id), this is the range of the sum, denoted by $(\G_{\inf}\backslash\G)_{<K}$, taken in (\ref{defn_R_Gamma}). With a slight abuse of notation, we use $\g \in \G_\inf\backslash \G$ to denote a sum over a representative from each coset.

To ensure that (\ref{defn_R_Gamma}) is well-defined we require to check that the summands are invariant under the replacement of $\g$ with $\pm T^n\g$ for $n\in\Z$. Obviously $\reg(T^n\g,\t)=\reg(\g,\t)$ for any $n$, and $\g$ and $-\g$ act in the same way on $\H$ as well as on the cusp representatives so the Rademacher regularisation factor has the required invariance. Left multiplication by $T$ does not change the lower row of $\g$ so the factor $\jac(\g,\t)^{\mathsmaller{1/4}}$ is invariant under the substitution of $T^n\g$ for $\g$. Using the explicit description (\ref{Dedmult}) for $\e$ we see that $\e(T^n\g)=\ex(-\tfrac{n}{24})\,\e(\g)$, so that the product $\e(\g)^{\mathsmaller{-3}}\ex(-\tfrac{\g\t}{8})$, is also invariant under the substitution of $T^n\g$ for $\g$, and then the identity $\e(-\g)^{\mathsmaller{-3}}\jac(-\g,\t)^{\mathsmaller{1/4}}=\e(\g)^{\mathsmaller{-3}}\jac(\g,\t)^{\mathsmaller{1/4}}$ completes the verification that every summand of (\ref{defn_R_Gamma}) is invariant under the replacement of $\g$ with $\pm T^n\g$.

While the regularisation factor might seem ad hoc at the first sight, here we would like to argue that it is in fact strictly necessary and very natural. 
First, to motivate the appearance of a regularisation note that a series 
\be\label{explainreg}
\sum_{\g\in\G_{\infty}\backslash\G}\psi(\g)\ex(\alpha\g\t)\jac(\g,\t)^{{\mathsmaller{w/2}}},
\ee
for $\alpha$ a real constant and $\psi$ a (compatible) multiplier system of weight $w$ on $\G$, will generally not converge (absolutely and uniformly on compacta) unless $w>2$. (In case it does converge it defines a modular form of weight $w$ on $\G$.) In \cite{Rad_FuncEqnModInv} Rademacher demonstrated a method to regularise this sum for $\psi\equiv 1$, $\alpha=-1$ and $w=0$ in the case that $\G=\SL_2(\Z)$. His prescription may be described as follows: multiplying each summand by the regularisation factor ${\rm r}^0(-1,\g,\t)$ which is $1$ when $\g$ is upper triangular and 
$$
\ex(\mathsmaller{\g\t-\g\infty})\ex(\g\infty-\g\t,1)=
\ex(\g\t-\g\infty)(\ex(\g\infty-\g\t)-1)
$$ 
otherwise, and replacing the sum over the coset space $\G_{\infty}\backslash\G$ with a limit of sums over rectangles $(\G_{\infty}\backslash\G)_{<K}$, we obtain
$$
\ex(-\t)+\lim_{K\to \infty}\sum_{\g\in(\G_{\infty}\backslash\G)_{<K}^{\times}}\ex(-\g\t)-\ex(-\g\infty)\;.
$$ 
He then went on to prove (cf. loc. cit.) that this expression converges and defines an $\SL_2(\Z)$ invariant function on the upper half-plane, which is nothing but the familiar $j$-function (up to an additive constant). More generally, in \cite{Duncan2009}, for example, it is shown that for $\psi\equiv 1$ and $\alpha,w/2\in\Z$, inclusion of the factor 
$$
{\rm r}^{\mathsmaller{w/2}}(\alpha,\g,\t)=\ex(\alpha(\g\infty-\g\t))\ex(\alpha(\g\t-\g\infty),1-w)
$$ 
for summands corresponding to cosets with upper triangular representatives, together with a limit of sums over rectangles, regularises the sum (\ref{explainreg}) for $\G$ a subgroup of $\SL_2(\R)$ that is commensurable with the modular group. Putting $\a=-1/8$ and $w=1/2$, we recover exactly our regularisation factor (\ref{defn:reg}). It is rather surprising that this straightforward generalisation of regularisation scheme does its job in regularising the sum in \eq{defn_R_Gamma}.

To understand the modular properties of $R_{\G}$, it is useful to consider the companion function $S_{\G}$ defined by
\be\label{defn_S_Gamma}
	S_{\G}(\t)=\lim_{K\to \infty}\sum_{\g\in(\G_{\infty}\backslash\G)_{<K}}
	\e(\g)^{3}
	\ex(\tfrac{\g\t}{8})
	\jac(\g,\t)^{3/4}.
\ee
According to the philosophy of the previous paragraph $S_{\G}$, supposing it converges, should have  modular transformations of weight $3/2$ for $\G$ with multiplier system coinciding with that of $\eta(\t)^3$. Also, we should have that $S_{\G}={\cal O}(q^{1/8})$ as $\t\to\inf$. It will develop in \S\ref{sec:Var} that $R_{\G}$ is a mock modular form and that $S_{\G}$ is its {shadow}. 

\begin{rmk}
The functions $\hat{\Sigma}(\t)$ and $\hat{\Sigma}^{\circ}(\t)$ of \cite{Eguchi2009a} are recovered by taking $\G=\SL_2(\Z)$ and $\G=\G_0(2)$, respectively, in $\hat{R}_{\G}(\t)$, up to a factor of $-2$. We will demonstrate presently that $R_{\G}(\t)=\Sigma(\t)/2$ for $\Sigma(\t)$ is as in loc. cit. when $\G=\G_0(1)$.
\end{rmk}

Apart from considering the Rademacher with $\psi(\g) =\e^{\mathsmaller{-3}}(\g)$, to construct the $M_{24}$ McKay--Thompson series we also need to extend the above construction to the cases with a more general multiplier system $\psi$.  In particular, we would sometimes like to consider different Rademacher sums with the same group $\G$ but different multipliers. 

Given a group morphism $\rho:\G\to\C^*$ that is trivial on $\G_{\inf}$ we may consider the following twists of $R_{\G}$ and $S_{\G}$:
\bea\notag
R_{\G,\rho}(\t)
&=&
\lim_{K\to \inf}\sum_{\g\in(\G_{\infty}\backslash\G)_{<K}}
	\psi(\g)
	\ex(-\tfrac{\g\t}{8})
	\reg(\g,\t)
	\jac(\g,\t)^{1/4}
	\\\label{defn:RGrho}
S_{\G,\rho}(\t)
&=&
\lim_{K\to \inf}\sum_{\g\in(\G_{\infty}\backslash\G)_{<K}}
	\bar \psi(\g)
		\ex(\tfrac{\g\t}{8})
	\jac(\g,\t)^{3/4}
\eea
where we have $\psi(\g)=\r(\g)\e(\g)^{\mathsmaller{-3}}$. In this paper we will take $\G=\G_0(n)$ for some $n\in\Z^+$ and $\r=\rho_{n|h}$ given in 
\eq{rho_multiplier} for $h$ simultaneously dividing $n$ and $24$. 
Not surprisingly, later on we are going to take $(n,h)$ to be the $(n_g,h_g)$ coming from the cycle shapes of $M_{24}$ as given by Table \ref{examples_eta}. 

To ease notation, we write $R_{n|h}$ for $R_{\G,\rho}$ when $\G=\G_0(n)$ and $\rho=\rho_{n|h}$, and we apply the directly analogous interpretation to the  $S_{n|h}$. 
\vspace{.1cm}
The main result in this article is the following.
\begin{thm}\label{thm:HisR}
Let $g\in M_{24}$. Then $H_g=-2R_{n|h}$  when $n=n_g$ and $h={h_g}$.
\end{thm}

The rest of the paper is devoted to the proof of the above theorem.

\section{Convergence}
\label{sec:Conv}
\setcounter{equation}{0}

After defining the Rademacher sums formally in the last section, in this section we would like to discuss the convergence of the sums.
An integral step towards Theorem \ref{thm:HisR} is to prove its convergence. It is a rather subtle issue and in this section we will establish it conditional upon a hypothesis whose validity will be established in \S  \ref{Spec}.
To be more precise, the aim in this section is to demonstrate that if the Selberg--Kloosterman zeta function (\ref{defn:SelKlozeta}) attached to $\psi=\rho\e^{\mathsmaller{-3}}$ and $\G=\G_0(n)$ converges at $s=3/4$ for some character $\rho$ on $\G$ then the limits (\ref{defn_R_Gamma}) and (\ref{defn_S_Gamma}) defining the Rademacher sums $R_{\G,\rho}$ and $S_{\G,\rho}$, respectively, define holomorphic functions on the upper half-plane.

Our arguments in this section owe a lot to the work of Niebur in \cite{Nie_ConstAutInts}. We cannot apply his results verbatim since he worked with (mock) modular forms of negative weight and ours have positive weight, but  many of the procedures we employ have a counterpart in loc. cit.

Let $\G=\G_0(n)$ for some positive integer $n$ and $\psi(\g)=\rho(\g)\e(\g)^{\mathsmaller{-3}}$ for some character $\rho$ on $\G$ satisfying $\r(\g) = \r(T\g)=\r(\g T)$. 
According to the definitions of $R_{\G,\rho}$ and $S_{\G,\rho}$, we can separate the contribution from the trivial coset and the rest in the following way 
\begin{align*}
R_{\G,\rho}(\t)&=q^{-1/8}+\lim_{K\to \inf}\sum_{\g\in(\G_{\inf}\backslash\G)^{\times}_{<K}}R_{\g}(\t)\\
S_{\G,\rho}(\t)&=q^{1/8}+\lim_{K\to \inf}\sum_{\g\in(\G_{\inf}\backslash\G)^{\times}_{<K}}S_{\g}(\t)\;.
\end{align*}
 Using the expression of the (generalised) exponentials as  infinite sums and isolating the first term in the sum,  we obtain $R_{\g}(\t)=R_{\g}^0(\t)+R_{\g}^+(\t)$ and $S_{\g}(\t)=S_{\g}^0(\t)+S_{\g}^+(\t)$, where
\begin{align}\label{defn:RS0gamma}
R_\g^0 &= e(\tfrac{1}{8})\,\psi(\g)\ex(-\tfrac{\g\inf}{8}) \,c^{\mathsmaller{-3/2}} \,(\t-\g^{-1}\inf)^{-1}\\ \notag
S_\g^0&=\bar\psi(\g) \ex(\tfrac{\g\inf}{8}) \,c^{\mathsmaller{-3/2}} \,(\t-\g^{-1}\inf)^{-3/2}
\end{align}
 and 
\begin{align}\label{defn:RSpgamma}
R_\g^+ &=\psi(\g)\ex(-\tfrac{\g\inf}{8})   \sum_{m=1}^\inf \frac{c^{-2m-3/2} (\t-\g^{-1}\inf)^{-m-1}({2\p i}/{8})^{m+1/2}}{\G(m+3/2)} \\ \notag
S_\g^+&=\bar\psi(\g) \ex(\tfrac{\g\inf}{8})\sum_{m=1}^\inf \frac{c^{-2m-3/2} (\t-\g^{-1}\inf)^{-m-3/2}(-{2\p i}/{8})^{m}}{m!} \;.
\end{align}

We have also applied the identities
\begin{align*}
\jac(\g,\t)=c^{-2}(\t-\g^{-1}\inf)^{-2}\quad,\quad \g\inf-\g\t=c^{-2}(\t-\g^{-1}\inf)^{-1}
\end{align*}
in rewriting the above expressions.
With these definitions, we have the following:

\begin{lem}\label{lem:sumRpgammaconv}
The sums $\sum_{\mathsmaller{\g\in(\G_{\inf}\backslash \G)^{\times}}}R_{\g}^+(\t)$ and $\sum_{\mathsmaller{\g\in(\G_{\inf}\backslash \G)^{\times}}}S_{\g}^+(\t)$ are absolutely convergent, locally uniformly for $\t\in\H$.
\end{lem}
\begin{proof}
We give details for the first sum. A very similar argument takes care of the second one. 

By (\ref{defn:RSpgamma}) we have that $R_{\g}^+(\t)={\cal O}(c^{\mathsmaller{-7/2}}\,|\t-\g^{-1}\inf|^{-2})$ as $c\to\inf$. 
For $\G=\G_0(n)$,  the representatives of the non-trivial cosets in the coset space $\G_{\inf}\backslash\G$ are indexed by their lower rows which are tuples of two integers $(c,d)$ with $c>0$ and $n|c$, $\gcd(c,d)=1$.
A procedure we will repeatedly use is to split the sum over $d$ into the following double sum
\be\label{eqn:sumgtosumcdn}
\sum_{\g\in(\G_{\inf}\backslash\G)^{\times}}f(c,d)=
\sum_{\substack{0<c\\n|c}}
\sum_{\substack{0 \leq d'<c\\(c,d')=1}}
\sum_{\ell\in\Z}
f(c,d'+\ell c)
\ee
for an arbitrary function $f(x,y)$.
Notice that the above sum over the tuple $(c,d')$ can be viewed as the sum over a complete and irredundant set of representatives of a double coset in $\G_\inf\backslash\G/\G_\inf$.

Taking $f(x,y)=x^{-7/2}|\t+y/x|^{-2}$ we find that
$$\sum_{\g\in(\G_{\inf}\backslash \G)^{\times}}|R_{\g}^+(\t)|
\ll
\sum_{\substack{0<c\\n|c}}c^{-7/2}
\sum_{\substack{0\leq d'<c\\(c,d')=1}}
\sum_{\ell\in\Z}|\t+d'/c+\ell|^{-2}.
$$
Since there are not more than $c$ invertible elements in the ring $\Z/c$, we conclude that the sum $\sum_{\mathsmaller{\g\in(\G_{\inf}\backslash \G)^{\times}}}R_{\g}^+(\t)$ is absolutely convergent, locally uniformly for $\t\in\H$, as required.
\end{proof}

Hence, in order to show the convergence of the sums defining $R_{\G,\rho}(\t)$ and $S_{\G,\rho}(\t)$ it now suffices to establish the convergence of the limits
\be\label{eqn:LimSum1}
\lim_{ K\to \inf}\sum_{\g\in(\G_{\inf}\backslash\G)^{\times}_{<K}}R_{\g}^0(\t),\quad
\lim_{ K\to \inf}\sum_{\g\in(\G_{\inf}\backslash\G)^{\times}_{<K}}S_{\g}^0(\t),
\ee
and apply Lemma \ref{lem:sumRpgammaconv}.

After carefully examining the limit and employing the Lipschitz summation formula  in the form of Lemma \ref{lem:LipSumAnlg}, the first of the above sums can be rewritten as 
\begin{align}\label{eqn:LimSum7}
&\lim_{K\to \inf}\sum_{\mathsmaller{\g\in(\G_{\inf}\backslash\G)^{\times}_{<K}}}R_{\g}^0(\t) = 2\p \ex(-\tfrac{1}{8})\sum_{k=1}^\inf \ex((k-\tfrac{1}{8})\t) 
\lim_{K\to \inf} \sum_{\substack{0<c< K\\n|c}}c^{-3/2} S(0,k,c,\psi) \;,
 \end{align}
where we have used the {\em generalised Kloosterman sum} (cf. \cite{GolSar_Kloo}) for a given group $\G$:
\be\label{defn:genKlosum}
S(m,\ell,c,\xi )=\sum_{\substack{\g\in\G_{\inf}\backslash\G/\G_\inf\\c(\g)=c}}\xi (\g)\ex((m-\alpha)\g\inf)\ex(-(\ell-\alpha)\g^{-1}\inf)
\ee
for $\ell\in\Z$, where for $m,\ell,c\in\Z$ with $0<c$ and $\xi $ a multiplier system on $\G$, the value $0\leq \alpha<1$ is determined by $\xi $ according to the requirement that $\xi (T)=\ex(\alpha)$.
An important object closely related to the generalised Kloosterman sum is the {\em Selberg--Kloosterman zeta function} of $\G$ (cf. loc. cit.), defined as
\be\label{defn:SelKlozeta}
Z_{m,\ell}(s,\psi)=\sum_{c>0}\frac{S(m,\ell,c,\psi)}{c^{2s}}\;.
\ee
The details of the manipulation leading to \eq{eqn:LimSum7} are given in Appendix \ref{details_convergence}.

Hence, as promised, if we assume the convergence of $Z_{0,k}(s,\r \e^{\mathsmaller{-3}})$ at $s=3/4$ we will have shown the convergence of the first sum in \eq{eqn:LimSum1}, and hence $R_{\G,\r}$ upon using  Lemma \eq{lem:sumRpgammaconv}.
The convergence of the Selberg--Kloosterman zeta function will be demonstrated in \S \ref{Spec}. Finally, a similar but less subtle argument also establishes the convergence of the second sum in \eq{eqn:LimSum1}. In particular, in the case of $S_\g^0$ the application of the Lipschitz summation formula is less delicate: We can apply (\ref{eqn:Lipsum}) with $s=3/2, \a=7/8$ and the error term $E_K$ is absent. 
Taken together we have the following: 

\begin{prop}\label{prop:convRS}
The expression (\ref{defn_R_Gamma}) defining $R_{\G,\rho}(\t)$ converges, locally uniformly for $\t\in\H$, thus defining a holomorphic function on $\H$. Also, the expression (\ref{defn_S_Gamma}) defining $S_{\G,\rho}(\t)$ converges, locally uniformly for $\t\in\H$, thus defining a holomorphic function on $\H$.\end{prop}

In practice it is useful to know that there is some flexibility in the sum over ``rectangles'' in the definition of $R_{\G,\rho}$ and $S_{\G,\rho}$ (cf. (\ref{defn_R_Gamma}) and (\ref{defn_S_Gamma})). For example, given $\sigma\in\G$ we may consider the sums
\be\label{eqn:RSsumsimgsrecs}
	\lim_{K\to\inf}\sum_{\g\in(\G_{\inf}\backslash\G)_{<K}\sigma}R_{\g}(\t),\quad
	\lim_{K\to\inf}\sum_{\g\in(\G_{\inf}\backslash\G)_{<K}\sigma}S_{\g}(\t),
\ee
where $(\G_{\inf}\backslash\G)_{<K}\sigma$ denotes the ``parallelogram" that is the image of $(\G_{\inf}\backslash\G)_{<K}$ under right multiplication by $\sigma$. The following result can be proved by using the technique employed in Lemma 4.3 of \cite{Nie_ConstAutInts} to extend the convergence arguments appearing already in this section. We suppress the details.
\begin{lem}\label{lem:convimgsrecs}
Let $\sigma\in\G$, and let $R_{\g}(\t)$ and $S_{\g}(\t)$ be as in (\ref{defn:RS0gamma}) and (\ref{defn:RSpgamma}), respectively. Then the expressions in (\ref{eqn:RSsumsimgsrecs}) converge uniformly for $\t\in\H$, thus defining holomorphic functions on $\H$. Moreover these functions coincide with $R_{\G,\rho}(\t)$ and $S_{\G,\rho}(\t)$ respectively.
\end{lem}

\section{Coefficients}
\label{sec:Coeffs}
\setcounter{equation}{0}

A common motivation to  study Rademacher sums is to obtain an exact expression for the Fourier coefficients of a given infinite $q$-series. 
In this section we will give such expressions for the functions $R_{\G,\r}$ and $S_{\G,\r}$. They are useful, for instance, in the physical study of the growth \eq{growth} of the twisted indices of supersymmetric states---the (twisted) entropy---as the energy of the states gets large. 

In the last section we have seen  that, by splitting the sum over representatives of cosets $\G_\inf\backslash\G$ into a double sum over those of the double cosets $\G_\inf\backslash\G/\G_\inf$ and over all integers $\ell\in \Z$, and applying the Lipschitz summation formula \eq{eqn:Lipsum1}, we obtain expressions \eq{eqn:LimSum7} for the objects \eq{eqn:LimSum1} as an infinite $q$-series. In this section we would like to show that the same can be done for the full $R_{\G,\r}, S_{\G,\r}$, and write down explicit expressions for the infinite $q$-series that the infinite sums \eq{defn:RGrho} converge to. 

Following similar steps as delineated in the last section, this time with the help of equation \eq{eqn:Lipsum}, it is not difficult to prove that the functions $R_{\G,\r}, S_{\G,\r}$ can be written as $q$-series with the following Fourier coefficients 

\begin{align}
\notag
&c_{\G,\rho}(k-\tfrac{1}{8})=\frac{2\pi\,\ex(-\tfrac{1}{8})}{(8k-1)^{1/4}}
\sum_{\substack{0<c\\n|c}}
\frac{1}{c}I_{1/2}\left(\frac{\pi}{2c}(8k-1)^{1/2}\right)
S(0,k,c,\psi)\\
\label{eqn:cstar_Gamma_SKz}
&c_{\G,\rho}^*(k+\tfrac{1}{8})=\ex(-\tfrac{3}{8})\,
{2\pi}{(8k+1)^{1/4}}
\sum_{\substack{0<c\\n|c}}
\frac{1}{c}J_{1/2}\left(\frac{\pi}{2c}(8k+1)^{1/2}\right)
S(1,k+1,c,\bar \psi)\;,
\end{align}
where $S(m,\ell,c,\psi)$ again denotes the generalised Kloosterman sum \eq{defn:genKlosum} and $\psi = \r \e^{\mathsmaller{-3}}$.

First note that the $q$-series with the above coefficients indeed define holomorphic functions on ${\mathbb H}$. To start with, again assuming the convergence of the Selberg--Kloosterman zeta function $Z_{0,k}(\tfrac{3}{4},\psi)$, we will show that the sums (over $c$) in \eq{eqn:cstar_Gamma_SKz} converge: Observe that since $c^{-1}I_{1/2}(\tfrac{\pi}{2c} \sqrt{8k-1}\hspace{.06cm})$ is bounded by a constant times $c^{\mathsmaller{-3/2}}$ for $c$ sufficiently large (cf. (\ref{BesselIJSmall})), the convergence of the right hand side of (\ref{eqn:cstar_Gamma_SKz}) is guaranteed by the convergence of $Z_{0,k}(\tfrac{3}{4},\psi)$. We establish the convergence of $Z_{0,k}(\tfrac{3}{4},\psi)$ for $k\in\Z$ in \S\ref{Spec} for the case that $\G=\G_0(n)$ and $\rho=\rho_{n|h}$ for some $h$ dividing both $n$ and $24$. From the above argument we see that this result implies the convergence of both the $c_{\G,\rho}(k-\tfrac{1}{8})$ for $k>0$ and the $c^*_{\G,\rho}(k+\tfrac{1}{8})$ for $k\geq 0$.

Once armed with the convergence of the terms $c_{\G,\rho}(k-\tfrac{1}{8})$ and $c_{\G,\rho}^*(k+\tfrac{1}{8})$ we may consider the following series 
$$
\sum_{k\geq 1}c_{\G,\rho}(k-\tfrac{1}{8})q^{k-1/8},\quad
\sum_{k\geq 0}c_{\G,\rho}^*(k+\tfrac{1}{8})q^{k+1/8}.
$$ 
The Bessel function $I_{1/2}(x)$ is asymptotic to $e^x/\sqrt{2\pi x}$ for large $x$ (cf. \S\ref{sec:Bessel}) so $c_{\G,\rho}(k-\tfrac{1}{8})$ is dominated by the first term in the summation over $c$ in (\ref{eqn:cstar_Gamma_SKz}) for $n$ sufficiently large, and similarly for $c_{\G,\rho}^*(k+\tfrac{1}{8})$; in the case that $\G=\G_0(n)$ this is the term with $c=n$. We  then have
\be\label{growth}
c_{\G,\rho}(k-\tfrac{1}{8})={\cal O}\left( \frac{e^{\pi\sqrt{8k-1}/2n} }{\sqrt{8k-1}}\right)
\ee
as $k\to \infty$ and a similar estimate holds for $c_{\G,\rho}^*(k+\tfrac{1}{8})$. In particular, we may conclude that the series $\Phi_{\G,\rho}$ and $\Phi_{\G,\rho}^*$, defined by
\begin{gather}
\label{eqn:F_Gamma}
\Phi_{\G,\rho}(\t)=q^{-1/8}+\sum_{k\geq 1}c_{\G,\rho}(k-\tfrac{1}{8})q^{k-1/8},\\
\label{eqn:Fstar_Gamma}
\Phi_{\G,\rho}^*(\t)=q^{1/8}+\sum_{k\geq 0}c_{\G,\rho}^*(k+\tfrac{1}{8})q^{k+1/8},
\end{gather}
converge absolutely and locally uniformly for $\t\in \H$ upon identifying $q=\ex(\t)$. 

We summarise the discussion of the previous two paragraphs in the following lemma.
\begin{lem}\label{lem:Zconvmodcconv}
The expressions (\ref{eqn:cstar_Gamma_SKz}) defining $c_{\G,\rho}(k-\tfrac{1}{8})$ and $c_{\G,\rho}^*(k+\tfrac{1}{8})$ converge. Further, the generating series (\ref{eqn:F_Gamma}) and (\ref{eqn:Fstar_Gamma}) for the $c_{\G,\rho}(k-\tfrac{1}{8})$ and $c_{\G,\rho}^*(k+\tfrac{1}{8})$ converge absolutely and locally uniformly for $\t\in\H$, thus defining holomorphic functions $\Phi_{\G,\rho}(\t)$ and $\Phi_{\G,\rho}^*(\t)$ on $\H$. These functions coincide with $R_{\G,\rho}(\t)$ and $S_{\G,\rho}(\t)$ respectively.
\end{lem}

\section{Variance}
\label{sec:Var}
\setcounter{equation}{0}

After establishing the convergence of the Rademachers sums $R_{\G,\rho}$ and $S_{\G,\rho}$, in this section we wish to determine how they transform under the action of $\G$. In particular, we will show that $S_{\G,\rho}$ is a weight $3/2$ cusp form and $R_{\G,\rho}$ is a weight $1/2$ mock modular form on $\G$. Moreover, the shadow of $R_{\G,\rho}$ is given by $S_{\G,\rho}$.

As before, we have $\G=\G_0(n)$ for some positive integer $n$, and let $\rho=\rho_{n|h}$ (cf. (\ref{rho_multiplier})) for some $h$ dividing both $n$ and $24$ and set $\psi=\rho\e^{-3}$. The variance of the Rademacher sum $S_{\G,\rho}$ can be easily established, and we have 

\begin{prop}\label{prop:SGammacuspform}
The function $S_{\G,\rho}(\t)$ is a cusp form of weight $3/2$ for $\G$ with multiplier system $\overline{\psi}$.
\begin{proof}
Since $\e^3$ is the multiplier of the cusp form $\h^3(\t)$ and $\r$ is a group character satisfying $\r(\g)\r(\sigma) = \r(\g\sigma)$ for all $\g,\sigma \in \G_0(n)$, we see that $\bar \psi=\r^{-1}\e^3$ is a weight $3/2$ multiplier system, which leads immediately to 
\be\label{variance_S_gamma}
 S_{\g}|_{\overline{\psi},\mathsmaller{3/2}}\sigma=S_{\g\sigma} 
 \ee  
 for $\s\in\G$, where $S_\g$ is as defined in \eq{defn:RS0gamma}-\eq{defn:RSpgamma}. Moreover, when $h=1$ and hence $\r$ is trivial, this also holds more generally for $\s\in \SL_2(\Z)$. 

Applying Lemma \ref{lem:convimgsrecs} to $S_{\G,\rho}=\lim_{\mathsmaller{K\to\inf}}\sum_{{\g\in(\G_{\inf}\backslash\G)_{\mathsmaller{<K}}}}S_{\g}(\t)$ and using \eq{variance_S_gamma}, we immediately have $S_{\G,\rho}|_{\overline{\psi},\mathsmaller{3/2}}\sigma= S_{\G,\rho}$ for  $\s\in\G$. This establishes that $S_{\G,\rho}$ is a modular form of weight $3/2$ for $\G$ with multiplier system $\overline{\psi}$.
We also need to show that $S_{\G,\rho}$ vanishes at all cusps. 
For the cusp at $\inf$, this is clear from  Lemma \ref{lem:Zconvmodcconv}. We will now show that it is also true for the other cusps of $\G$. 

An important property of a given cusp is its width. 
Suppose $\k\in \QQ$ is a representative of a cusp of $\G \subset \SL_2(\Z)$ that is related to the infinite cusp by $\s\inf = \k$, $\s \in \SL_2(\Z)$. We define the {\it width} $u$ of the cusp $\k$ to be the positive integer such that the subgroup of $\G$  stabilising $\k$ is given by
\be\label{width}
\Big\{ \g \in \G\Big\lvert \g \k = \k \Big\} = \s \langle \pm T^u \rangle\s^{-1} = \s \G_\inf^u \s^{-1}\;.
\ee

Alternatively, we can rephrase the above definition using the concept of a `scaling element' which will become important later. 
Consider an element $  \til \s \in \SL_2(\R)$, to be called a {\it scaling element} of $\G$ at cusp $\k$, such that $\k=\til \s\inf$ and the subgroup of $\G$  stabilising $\k$ is given by $\til \s \G_\inf  \til \s^{-1}$ (Cf. \cite[\S2.6]{Duncan2009}).  Such a scaling element necessarily has the form 
\be\label{scaling}
\til \s=  \s   U  T^\b\quad,\quad \rm{with}\quad  \s \in \SL_2(\Z)\;,\;  \b \in \QQ\;,
\ee
where $U= \big(\begin{smallmatrix} \sqrt{u} & 0 \\ 0& 1/\sqrt{u} \end{smallmatrix}\big)$ maps $\t \to u\t$ and $T^\b =  \big(\begin{smallmatrix} 1 & \b \\ 0& 1 \end{smallmatrix}\big)$.

Given a cusp representative $\k = \s \inf $ of width $v$ of $\G_0(n)$, in order to study the behaviour of $S_{n|h}$ at $\k$, it is convenient to consider the function $S_{n|h}^{(\k)} (\t)$ which we will now define. 
For a given scaling element $\til \s$, for $h=1$ we define 
\begin{align}\notag
S_{n|1}^{(\k)} (\t)  = S_{n\lvert 1}\big\lvert_{\e^3,\mathsmaller{3/2}} \s (v(\t+\b)\hspace{.02cm})  =  \e^3(\s) \jac(\s,v(\t+\b))^{\mathsmaller{3/4}}  S_{n\vert h}(\til\s\t)\;.
\end{align}
From \eq{variance_S_gamma} we get 
\[ S_{n|1}^{(\k)} (\t)  = \lim_{\mathsmaller{K\to\inf}}\sum_{\g \in(\G_{\inf}\backslash\G)_{\mathsmaller{<K}} \s} S_\g (v(\t+\b)\hspace{.02cm}) \;.
\]
 Splitting the above sum into a sum over the double coset $(\G_{\inf}\backslash\G)_{\mathsmaller{<K}} \s/\G_\inf^v$ and applying the Lipschitz summation formula as before, we obtain that $S_{n|1}^{(\k)} (\t) $ takes the form of an infinite series $\sum_{k=0}^\inf c_{\G,\r}^{*(\k)}(k+\n) \,\ex( (k+\n) {\t})$ where $0< \n\leq1$ differs from $\tfrac{v}{8}$ by an integer. This proves that $S_{n\vert1}$ vanishes at  $\k$. 

More generally, for  $h\neq 1$ we  again consider
\begin{align}
S_{n\vert h}\big\lvert_{\e^3,\mathsmaller{3/2}} \s (\t) & = \e^3(\s) \jac(\s,\t)^{\mathsmaller{3/4}}  S_{n\vert h}(\s\t)\\\notag
&=  \lim_{\mathsmaller{K\to\inf}}\sum_{\g \in(\G_{\inf}\backslash\G)_{\mathsmaller{<K}} \s}  \r^{\mathsmaller{-1}}({\g\s^{-1}}) \e^3(\g) \jac(\g,\t)^{\mathsmaller{3/4}}  \ex(\tfrac{\g\t}{8})\;.
\end{align}
If $\k$ has width $v$ in $\G_0(n)$ and has width $u$ in $\G_0(nh)$, we can define for a given scaling element $\til\s$
\[S_{n\vert h}^{(\k)}(\t) = S_{n\vert h}\big\lvert_{\e^3,\mathsmaller{3/2}} \s\, (u(\t+\b)\hspace{.02cm})  \;.
\]
To obtain its expression as an infinite series, define $0\leq m <h$ such that $\r(\s T^v\s^{-1}) =\ex(\tfrac{m}{h})$ and 
 proceed as before, ultimately obtainin the  general expression 
 \be\label{S_nh_other_cusp}
S_{n\vert h}^{(\k)}(\t) =\sum_{k=0}^\inf c_{\G,\r}^{*(\k)}(k+\n) \ex( (k+\n) {\tfrac{u}{v}\t})
\quad, \quad 0< \n \leq1 \;. \ee
More precisely, $\nu=\{\!\{ \tfrac{m}{h}+\tfrac{v}{8}\}\!\}$ is given by the sawtooth function $\{\!\{x\}\!\}$ defined as the difference between a real number and the largest integer that is smaller than it: 
\be\notag
\{\!\{x\}\!\} = \begin{cases} x- \lfloor x \rfloor  & x\in\R\backslash \Z \\ 1 & x\in\Z \end{cases}
\ee
  where the floor function $\lfloor x \rfloor$ denotes the largest integer not exceeding $x$.  
This shows that $S_{n|h}(\t)$ vanishes at all cusps of $\G_0(n)$ and thereby finishes the proof. 
\end{proof}
\end{prop}

Our next task is to establish the variance of the Rademacher sum $R_{\G,\rho}$.   
\begin{prop}\label{prop:Rismock}
The function $R_{\G,\rho}$ is a mock modular form of weight $1/2$ on $\G$ with multiplier system $\psi$ and with shadow $S_{\G,\rho}$.
\end{prop}
\begin{proof}

First  observe that $R_\g = \psi(\g) \ex(-\tfrac{\g\t}{8}) \jac(\g,\t)^{1/4} {\rm reg}(\g,\t)$ would transform nicely if it weren't for the regularisation factor. Therefore, it would be useful to separate from $R_\g$ the ``unregularised'' part 
\begin{gather}\label{defn:Pgamma}
	P_{\g}(\t)=\psi(\g)\ex(-\g\t/8)\jac(\g,\t)^{1/4},
\end{gather}
for $\g\in \G$, and consider
\[
R_{\g}(\t)-P_{\g}(\t)
	=
	\psi(\g)
	\ex(-\tfrac{\g\inf}{8})\jac(\g,\t)^{1/4}
	\left(
	\ex(X,1/2)-\ex(X)
	\right)
	\]
with $X= \tfrac{\g\inf-\g\t}{8}$. 
Using the definition \eq{defn:genexpintegral}, we have
\[ \ex(X,s )- \ex(X) = -\frac{\ex(X)}{\G(s)} \int_{2 \p i X}^\inf  e^{-t} t^{s-1}\,dt\;.\]  
Setting $t=\tfrac{2\p i}{8}(\g z-\g\t)$ in the integral and using the identity
\be\label{eqn:jacjac} 
	\jac(\g,z)\jac(\g,\t)(z-\t)^{2}=(\g z-\g \t)^2
\ee
we get 
\[ \ex(X,1/2 )- \ex(X) =\frac{1}{2}\ex(\tfrac{1}{8})\ex(\tfrac{\g\inf}{8}) \,{\rm jac}(\g,\t)^{\mathsmaller{-1/4}} \int^\inf_{\g^{-1}\inf} \ex(-\tfrac{\g z}{8}) \, {\rm jac}^{\mathsmaller{3/4}}(\g,z) \,(z-\t)^{\mathsmaller{-1/2}} \,dz\;.
\]

It will be convenient to introduce an integral operator $J_x$ for $x$ in the extended upper half-plane ($\H \cup \QQ\cup \inf$)
which acts on a holomorphic function $g$ on $\H$ (with sufficiently rapid decay as $\Im(z)\to \inf$) as 
\be\label{defn:Jw}
	{\left(J_xg\right)(\t)}
	=\frac{\ex(\tfrac{1}{8})}{2}
	\int_{\bar x}^{\inf}
	\bar g(z)
	(z-{\t})^{-1/2}
	{\rm d}z\;.
\ee

From the definition \eq{def_mock}, we see that we have a weight 1/2 mock modular form with shadow $g$ if $\hat h = h + J_\t g$ transforms as a weight 1/2 modular form. 

Comparing with the definition of the integral operator $J_x$, we have 
$$
R_{\g}(\t)-P_{\g}(\t)=\left(J_{\g^{{-1}}\inf}S_{\g}\right)(\t)\;.
$$
Note in particular that this is trivially true when $\g$ is in the trivial coset given by upper triangular matrices, since both sides of the equation vanish. 

By the fact that $\psi$ is a multiplier system for $\G$ of weight $1/2$, we immediately have for an element $\s$ of $\G$ \[P_\g\lvert_{\psi,\mathsmaller{1/2}} \s= P_{\g\s}\,,\;{\rm and }\;\;
 (J_{\g^{-1}\inf}S_\g)\lvert_{\psi,\mathsmaller{1/2}} \s= (J_{\s^{-1}\g^{-1}\inf }- J_{\s^{-1}\inf })S_{\g\s}\;\;.\]
 Upon summation over the representatives of the rectangle $(\G_\inf\backslash\G)_{<K}, K\to \inf$ and evoking  Lemma \ref{lem:convimgsrecs}, we obtain 
 \be\label{eqn:RslashJS}
	R_{\G,\r}|_{\psi,\mathsmaller{1/2}}\,\s
	=
	R_{\G,\r}
	-J_{\s^{-1}\inf}S_{\G,\r}\;.
\ee

Now it is an easy calculation to show that\footnote{In fact, one can  show that the non-holomorphic completion $R_{\G}+\tilde{S}_{\G}$ also has a Rademacher form given by \be
\hat{R}_{\G}(\t)=\lim_{K\to \infty}\sum_{\g\in(\G_{\infty}\backslash\G)_{<K}}
	\e(\g)^{\mathsmaller{-3}}
	\ex(-\tfrac{\g\t}{8})
	\com(\g\t)
	\jac(\g,\t)^{1/4}
\ee
where $\com(\t)$ is the {\em non-holomorphic completion factor} given in terms of the incomplete Gamma function, $\G(\alpha,x)=\int_x^{\infty}t^{\alpha-1}e^t{\rm d}t$, by
$$
\com(\t)=1-\G(1/2,\pi\Im(\t)/2)/\sqrt{\pi}\;.
$$ 
}
$$
	\left(R_{\G,r} + J_\t S_{\G,\r}\right)\Big|_{\psi,\mathsmaller{1/2}}\,\s=R_{\G,r} + J_\t S_{\G,\r}\;,
$$
which by definition \eq{def_mock} proves the claim.
\end{proof}

\section{Coincidence}
\label{sec:Coin}
\setcounter{equation}{0}

Recall from \S\ref{sec:Radsums} that we write $R_{n|h}$ for $R_{\G,\rho}$ when $\G=\G_0(n)$ and $\rho=\rho_{n|h}$, and similarly for $S_{n|h}$. 
In \S\ref{sec:Var} we have established the transformation property of $R_{n|h}$ and $S_{n|h}$ under the group $\G_0(n)$. In this section we will explicitly identify them. More precisely, we will first prove in  \textsection\ref{SSS} that 
\be\label{S_coincidence} S_{n|h} = \begin{cases} \l_n \h^3, \l_n \in \C& h=1 \\ 0 & h>1 \end{cases} \ee
for the $(n,h)=(n_g,h_g)$ that appear in the cycle shapes of $[g]\subset M_{24}$, (cf. Table \ref{examples_eta}),  and later in \textsection\ref{RRR} complete the proof our main theorem \ref{thm:HisR} by showing $H_g=-2R_{n|h}$, assuming the result of \S\ref{Spec}.

\subsection{Determining $S$}
\label{SSS}
Let $n$ and $h$ be positive integers such that $h$ divides both $n$ and $24$ and consider the function 
$$
F_{n|h}(\t)=\frac{S_{n|h}(\t)}{\eta(\t)^3}\;.
$$ 
Since $S_{n|h}$ has the same multiplier system as $\eta^3$ on $\G_0(nh)$ and $\eta$ does not vanish on the upper half-plane, $F_{n|h}$ is a modular function for $\Gamma_0(nh)$ that satisfies the transformation rule
\be\label{eqn:fnh_{n|h}xfm}
	F_{n|h}(\g\t)=\rho(\g)F_{n|h}(\t)
\ee
for $\g\in \G_0(n)$ where $\rho=\rho_{n|h}$ is given explicitly by \eq{rho_multiplier}. In particular, $F_{n|h}$ induces a morphism, which we denote by $f_{n|h}$, of Riemann surfaces $X_0(nh)\to\PP^1$ where $X_0(N)$ denotes the {\em modular curve} of level $N$ and is formed by taking the quotient of the {\em extended upper half-plane} 
\[
\H^*=\H\cup\QQ\cup\{\inf\}
\]
by the action of the group $\G_0(N)$. We will denote the corresponding canonical map by 
\be\label{defn:xtduhp}
\varphi_{N}: \H^*\to X_0(N)\;.
\ee

Note that there are no non-constant maps $X_0(nh)\to\C$, so either $F_{n|h}$ has a pole at some cusp of $\G_0(nh)$, or $S_{n|h}=\lambda_{n|h}\eta^3$ for some $\lambda_{n|h}\in\C$. Of course, $S_{n|h}$ and $\eta^3$ have different automorphy for $\G_0(n)$ when $h>1$, and an identity $S_{n|h}=\lambda_{n|h}\eta^3$ implies that $\lambda_{n|h}$ and hence $S_{n|h}$ are identically zero. 
Our main strategy to show \eq{S_coincidence} is to show that $F_{n|h}$ is bounded at all cusps of $\G_0(nh)$ for all the pairs $(n,h)$ under consideration.

Since the Fourier expansion of both $S_{n|h}$ and $\h^3$ around the infinite cusp starts with a term $q^{1/8}$(cf.  \eq{lem:Zconvmodcconv}\hspace{.1mm}), we see that $F_{n|h}(\t)$ is bounded as $\t \to i\inf$ and can only have poles at the cusps of $\G_0(nh)$ that are not represented by infinity. Furthermore, from the transformation \eq{eqn:fnh_{n|h}xfm} of $F_{n|h}$ under $\G_0(n)$, we see that $f_{n|h}$ has a pole at a given cusp of $\G_0(nh)$ if and only if it has a pole of the same order at every cusp that is an image of it under $\G_0(n)$. Therefore, to prove \eq{S_coincidence} it is sufficient to show that $F_{n|h}$ are bounded at all cusps of $\G_0(n)$ other than the infinite cusp. 

In order to analyse the behaviour of $F_{n|h}$ at the cusp representative $\k \in \QQ$, as before we will now consider its transformation under a given scaling element (cf. \eq{scaling})
$$\til \s = \s \bem \sqrt{u} & 0 \\ 0&1/ \sqrt{u}  \eem \bem 1 & \b \\ 0&1  \eem \quad,\quad \s \in \SL_2(\Z) \;,\; \b\in\R$$ for $\G_0(nh)$ at $\k =\til \s \inf$: 
\be\label{def_f_at_cusp}
F_{n|h}^{(\k)} (\t) =F_{n|h} (\til\s \t) = \frac{S_{n|h}(\til\s \t)}{\h^3(\til\s \t)} = \frac{S_{n|h}^{(\k)}(\t)}{  \h^3(u(\t+\b)\hspace{.02cm})} \;.
\ee
When $\til\s$ normalises $\G_0(nh)$, $F_{n|h}^{(\k)} (\t)  $ again defines a modular function for $\G_0(nh)$.
The order of poles or zeros of the function $F_{n|h}^{(\k)} (\t)$ at $\t\to i\inf$ gives the order of poles or zeros of the function $f_{n|h}: X_0(nh) \to \PP^1$  at the point $x$ given by the image of the $\k$ under the canonical map $\varphi_{nh}$. 

In the proofs, we will be especially interested in the scaling elements that are elements of the so-called Atkin-Lehner involutions $W_u$, which are defined for every exact divisor $u$ of $n$ as the set of matrices of the form
\[
\frac{1}{\sqrt{u}}\bem a u & b \\ c n & du \eem \quad,\quad a,b,c,d \in \Z
\]
with determinant 1. We say $u$ is an exact divisor of $n$ if $({n}/{u},u)=1$. An important element of $W_n$ is the {\it Fricke involution}, given by
\be \label{fricke}
w_n = \bem 0& -1/\sqrt{n} \\ \sqrt{n} &0  \eem \;.
\ee

From the expansion of $S_{n|h}^{(\k)}(\t)$ recorded in \eq{S_nh_other_cusp}, we see that $F_{n|h}^{(\k)} (\t)$ has the expansion  
\be\label{degree_pole}
F_{n|h}^{(\k)} (\t) =\sum_{k= 0}^\inf C_{n|h}(k)\, q^{\ell+\frac{u}{v}k}   \quad,\quad \ell=\tfrac{u}{v}\,\{\!\{ \tfrac{m}{h}+\tfrac{v}{8}\}\!\}-\tfrac{u}{8}\;.
\ee 
As before, $u$ denotes the width of the cusp in $\G_0(nh)$, $v$ denotes the width of the cusp in $\G_0(n)$, and $\ex(\tfrac{m}{h}) = \r(\s T^v\s^{-1})$ as defined in \eq{width}. 
Moreover, since $F_{n|h}^{(\k)} (\t) $ is invariant under $\langle \pm T\rangle\subset\til\s^{-1}\G_0(nh)\til\s$, its Fourier expansion can only involve integral powers of $q$ and it follows immediately that
\be\label{rough_degree_pole}
F_{n|h}^{(\k)} (\t)= {\cal O}(q^\ell)\quad,\quad \ell \geq 1+ \lfloor -\tfrac{u}{8} \rfloor \quad,\quad \ell \in \Z\;.
\ee
More generally, from the above formula it is also clear that, to prove that $f_{n|h}$ is bounded at $\k$ it is sufficient to prove that 
\be\label{nleq8condition1}
\{\!\{ \tfrac{m}{h}+\tfrac{v}{8}\}\!\} \geq \tfrac{v}{8}\;.
\ee

\begin{table}
 \resizebox{1.01\textwidth}{!}{
\begin{tabular}[H]{c|ccccccccccccccc}\toprule
$n$ & 1&2&3&4&5&6&7&8&10&11&12&14&15&21&23 
\\  \midrule
cusps & /& $0$& $0$& $0,\tfrac{1}{2}$& $0$ & $0,\tfrac{1}{2},\tfrac{1}{3}$& $0$& $0,\tfrac{1}{2},\tfrac{1}{4}$& $0,\tfrac{1}{2},\tfrac{1}{5}$& $0$ & $0,\tfrac{1}{2},\tfrac{1}{3},\tfrac{1}{4},\tfrac{1}{6}$& $0,\tfrac{1}{2},\tfrac{1}{7}$& $0,\tfrac{1}{3},\tfrac{1}{5}$& $0,\tfrac{1}{3},\tfrac{1}{7}$& $0$
\\\midrule  widths&/&2&3&4,1&5&6,3,2&7&8,2,1&10,5,2&11&12,3,4,3,1&14,7,2&15,5,3&21,7,3&23
 \\ 
\bottomrule
\end{tabular}}
\caption{\label{Cusps of G}\footnotesize{In this table we list the cusps other than the one at $i\inf$, which has width 1, of the groups $\G_0(n)$ for $n=n_g$ in Table \ref{examples_eta}. Note that, in these cases we have an extra (non-infinite) cusp, represented by $\tfrac{1}{q}$, for every divisor $q$ of  $n$, and its width is given by the smallest positive integer $v$ such that $n|q v^2$. Moreover, the corresponding multiplier $\r(\s T^v\s^{-1}) =\ex(\tfrac{m}{h})$ is given by $m = \tfrac{q^2v}{n}(1+qv)$ mod $h$.}
}
\end{table}

We shall first consider the cases $n\leq 8$ when there are no cusps of $\G_0(n)$ with width larger than $8$.  For these cases we have the following Lemma: 
\begin{lem}\label{lem:Spropeta1to8}
Consider the group $\G_0(n)$ with $1\leq n\leq 8$ and multiplier $\r_{n|h}$ defined in \eq{rho_multiplier} with $h$ dividing both $n$ and $24$. The pole-free condition \eq{nleq8condition1} is satisfied for all cusps of $\G_0(n)$ if and only if $(n,h)=(n_g,h_g)$ is one of the 13 pairs that correspond to the cycle shapes of $M_{24}$, as collected in Table \ref{examples_eta}.
\end{lem}
\begin{proof}

First consider the case $h=1$. In this case, since there are no cusps with $v>8$, we get from \eq{rough_degree_pole} that $f_{n|h}$ has no pole for all $1\leq n\leq 8$, $h=1$. 
 
For $h>1$, we need more refined information about the order of zeros of $S^{(\k)}_{n|h}$ to show that, for a given $n<9$, among all the possible values of $h$ that divide both $n$ and $24$,  the condition \eq{nleq8condition1} is satisfied if and only if $(n,h)=(n_g,h_g)$ is one of the 13 pairs that correspond to the cycle shapes of $M_{24}$. 

For the cusp with a representative $1/q$ and width $v$, taking $\s = \big(\begin{smallmatrix} 1&1\\ q&q+1 \end{smallmatrix}\big)$ we get 
\be\label{formula_for_m} m = \tfrac{q^2v}{n}(1+qv) \quad {\rm mod} \quad h\;. \ee 
In particular, we have $m=q$ and $v=n/q$ if $q$ is an exact divisor of $n$. For the zero cusp ($q=1$, note that 0 and 1 are related by $T\in\G_0(n)$), therefore, we need $h$ to satisfy  $\{\!\{ \tfrac{1}{h}+\tfrac{n}{8}\}\!\} \geq \tfrac{n}{8}$. It is straightforward to check that this condition is satisfied if and only if $(n,h)=(n_g,h_g)$ is one of the 13 pairs that show up in Table \ref{examples_eta}.

It remains to be shown that this condition is also satisfied at cusps other than $0$ or $\inf$. 
First we check that $\{\!\{ \tfrac{q}{h}+\tfrac{n/q}{8}\}\!\} \geq \tfrac{n/q}{8}$ is indeed satisfied for $n=h=6$, $q=2$ and $q=3$.
Finally, for $n=4$, we have $q=2$, $v=1$ and $m=3$ and we need to check $\{\!\{ \tfrac{3}{h}+\tfrac{1}{8}\}\!\} \geq \tfrac{1}{8}$, which is indeed the case for $h=4$.
\end{proof}

\begin{cor}

For $(n,h)=(n_g,h_g)$ being one of the 13 pairs with $1\leq n\leq 8$ that arise from in the cycle shapes of $M_{24}$, (cf. Table \ref{examples_eta}), we have $S_{n|1} = \l_{n} \h^3$ and $S_{n|h} =0$ for $h>1$.
\end{cor}

For $n>8$, a direct proof via \eq{nleq8condition1} is less straightforward and instead we would like to use arguments involving the topology of $X_0(n)$. When $X_0(nh)$ has genus one, we have the following two lemmas:
\begin{lem}\label{lem:Spropeta111415}
If $n\in\{11,14,15\}$ then $S_{n|1}=\lambda_{n}\eta^3$ for some $\lambda_{n}\in\C$.
\end{lem}
\begin{proof}

For these values of $n$ there is just one cusp of $\G_0(n)$ represented by $0$ which has width $v=n>8$. 
Hence, from \eq{rough_degree_pole} we see that the function $F_{n|1}^{(\k=0)}$ either has a simple pole at $0$ or is a constant. Suppose the former were true, then $\fno_{n|1}$ defines a degree one function from $X_0(n)$ to $\PP^1$, implying that $X_0(n)$ has genus zero as a Riemann surface. For $n\in\{11,14,15\}$ however, the genus of $X_0(n)$ is one, so we conclude that $\fno_{n|1}$ is constant in these cases.

\end{proof}

\begin{lem}
Suppose that $n\in\{10,12\}$ and $h=2$. Then $S_{n|h}=0$.
\end{lem}
\begin{proof}
The only cusp of $\G_0(20)$ that has width greater than $8$ is that represented by $0$, and similarly for $\G_0(24)$, so in either case the only possible pole of $f_{n|h}$ is at the image under $\varphi_{nh}:\H^*\to X_0(nh)$ of the zero cusp. 
The degree of the pole, if the pole exists, is given by \eq{degree_pole} with $2v=u=2n$ and $m=1$ equals 1, and so must be 1.  
In this case $f_{n|h}:X_0(nh)\to\PP^1$ is a degree one map, implying that $X_0(nh)$ has genus zero. But this is not the case for $nh=20$ or $nh=24$. We conclude that $f_{n|h}$ has no poles and is thus constant. The claim follows.

\end{proof}

For the remaining three classes with $(n,h) = (23,1), (21,3), (12,12)$, the modular curves $X=X_0(nh)$ have higher genera and we will use the Riemann-Roch theorem:
\be\label{eqn:RR}
	\dim K_X(D)=\dim\Omega_X(D)+\deg (D)+1-\gen (X)\;
\ee
to gain control over the allowed properties of the function $f_{n|h}$. Here $D$ denotes a divisor on the Riemann surface $X$, $K_X(D)$ denotes the vector space of meromorphic functions $f$ with the divisor $(f)$ given by its zeros and poles satisfying $(f)+D\geq 0$, and $\Omega_X(D)$ denotes the space of holomorphic differentials $\o$ on $X$ satisfying $(\omega)-D\geq 0$.  

\begin{lem}\label{lem:Spropeta23}
If $n=23$ then $S_{n|1}=\lambda_{n}\eta^3$ for some $\lambda_{n}\in\C$.
\end{lem}
\begin{proof}

The only cusp of $\G_0(23)$ with width larger than 8 is the one with width 23 represented by $0$, and hence the function $f_{23|1}$ can possibly have a pole only at $x \in X$ for $x$ the image of $0$ under $\varphi_{23}$ (cf. (\ref{defn:xtduhp})). 
More specifically, from \eq{rough_degree_pole} we see that $f_{23|1}$ is in the vector space $K_X(D)$ for $D=2x$. 
Since the genus of $X_{0}(23)$ is 2, according to the Riemann-Roch formula \eq{eqn:RR} we have $\dim K_X(D)=1+\dim\Omega_X(D)$. 
We would like to prove the claim by showing that $\dim\Omega_X(D)=0$ and hence $\dim K_X(D)=1$, and therefore  a function in $X(D)$ must be a constant function. 

To show there is no holomorphic differential with a double zero at $x$, let us consider first the divisor $D'=x$. We know that $K_{D'}(X)\cong \C$, or else there would be a meromorphic function on $X$ with a simple pole at $x$ and no other poles, hence an isomorphism between $X$ and ${\mathbb P}^1$, which is impossible given that the genus of $X$ is 2. From this we conclude from the Riemann-Roch formula that $\dim\Omega_X(D')= \dim K_{D'}(X) = 1$ and there is a one-dimensional space of holomorphic differentials $\o$ with $(\o) - x \geq 0$. 
Now such a holomorphic differential is given by the weight 2 cusp form $\h^2(\t)\h^2(23\t)$, which is a holomorphic differential with a simple zero at $x$ and a simple zero at the image of the infinite cusp under the canonical map $\H^*\to X_0(23)$ and no other zeros. 
From this we conclude $\dim\Omega_X(D)=0$ and  hence $\dim K_X(D)=1$. This completes the proof. 
\end{proof}

\begin{lem}\label{lem:n=21}
Suppose that $n=21$ and $h=3$. Then $S_{n|h}=0$.
\end{lem}
\begin{proof}
For $n=21$, by direct calculation one can check that the only cusp for which \eq{nleq8condition1} does not hold is the cusp represented by $0$. 
It corresponds to a cusp with width $63$ of $\G_0(63)$ and there are no other cusps of $\G_0(63)$ at which $f_{21|3}$ can have a pole. 
Moreover, from \eq{degree_pole}
for the zero cusp, with 
\[
\tfrac{63}{21} \{\!\{\tfrac{21}{8}+\tfrac{1}{3}\}\!\} - \tfrac{63}{8}=-5 \;
\]
equation \eq{degree_pole} shows that $f_{n|h}$ is in $K_D(X=X_0(63))$ with $D=5x$, $x= \varphi_{63}(0)$. 

The genus of $X_0(63)$ is $5$ so according to the Riemann--Roch theorem (cf. (\ref{eqn:RR})) we have $\dim K_X(D)=\dim\Omega_X(D)+1$, where $\Omega_X(D)$ denotes the space of holomorphic differentials on $X$ vanishing to order at least $5$ at $x$. We claim that $\dim \Omega_X(D)=0$. To show this, suppose $g\in S_2(\Gamma_0(63))$ corresponds to an element of $\Omega_X$ with a zero of order $5$ at $x$. Since the Fricke involution $w_{63}$
normalises $\G_0(63)$, it induces a linear automorphism of $S_2(\G_0(63))$. Then $g|_{1,2}\s$ is an element of $S_2(\G_0(63))$ that is ${\cal O}(q^6)$ as $\t\to i\inf$. By inspection there is no such cusp form of weight $2$ for $\G_0(63)$. We conclude that $\dim K_X(D)=1$. Since $K_X(D)$ includes constant functions, our $f_{n|h}$ is constant, and the desired result follows.

\end{proof}

\begin{lem}
Suppose that $n=h=12$. Then $S_{n|h}=0$.
\end{lem}
\begin{proof}
The group $\G_0(nh)=\G_0(144)$ has $9$ cusps with widths exceeding $8$: there are four with of width $9$, having representatives $1/4$, $3/4$, $1/8$ and $1/16$, there are three of width $16$, having representatives $1/3$, $2/3$ and $1/9$, and there is one of width $36$, with representative $1/2$, and one of width $144$, with representative $0$. The four cusps of width $9$ fuse under the action of $\G_0(n)=\G_0(12)$, as do the three cusps of width $16$.

First we would like to see at which of these cusps of $\G_0(144)$ the function $f_{n|h}$ could possibly have a pole. From the above data, we only need to look at the cusps of $\G_0(12)$. From the formula \eq{formula_for_m} we see that $f_{n|h}$ is either constant or has a pole at the image of the zero cusp, as the pole-free condition \eq{nleq8condition1} is satisfied at the all other cusps of $\G_0(12)$. 
 Moreover, putting $v=h=12$, $m=1$, $u=144$ in the equation \eq{degree_pole}, we see that $f_{n|h}$ can only have a pole of order 11 if it is not a constant.

To show that a pole of order 11 cannot happen, first we will compute $\dim K_X(D) =2$ for $X=X_0(144)$ and $D=11x$ with $x= \varphi_{144}(0)$. 
The genus of $X_0(144)$ is $13$ so according to the Riemann--Roch theorem (cf. (\ref{eqn:RR})) we have $\dim K_X(D)=\dim\Omega_X(D)-1$ where $\Omega_X(D)$ denotes the space of holomorphic differentials on $X$ vanishing to order at least $11$ at $x$. As in the proof of Lemma \ref{lem:n=21} we observe that elements of $\Omega_X(D)$ are in correspondence with cusp forms of weight $2$ on $\G_0(144)$ that are ${\cal O}(q^{12})$ as $\t \to i \inf$.  By inspection there is a three dimensional space of such cusp forms so $\dim K_X(D)=2$. 

Next set $D'=10x$. We claim that then $K_X(D)=K_X(D')$. 
For certainly $K_X(D')$ is a subspace of $K_X(D)$ and their dimensions coincide since the Riemann--Roch theorem implies $\dim K_X(D')=\dim\Omega_X(D')-2$ and $\dim\Omega_X(D')=4$ by inspection of $S_2(\G_0(144))$. 
We conclude that if $f_{n|h}$ is in $K_X(D)$ then it is also in $K_X(D')$ and then we must have $C_{n|h}(k=0)=0$ in \eq{degree_pole}, so that $f_{n|h}$ has no poles at any cusp and is therefore constant. This completes the proof.
\end{proof}

\begin{rmk}
 It is fascinating that the space $S_{3/2,\e^3}(\G_0(n))$ of cusp forms of weight $3/2$ on $\G_0(n)$ with multiplier system coinciding with that of $\eta^3$ is one-dimensional and spanned by $\eta^3$, whenever $n$ is the order of an element of $M_{23}$. This is demonstrated by the above Lemmas taken together, while the same is certainly not true for a generic positive integer $n$. 

For $n=9$ and $n=10$, for instance, it's possible that $F_{n|1}=S_{n|1}/\eta^3$ could define a degree one map $X_0(n)\to \PP^1$, but this is no contradiction since $X_0(n)$ has genus zero for $n=9$ and $n=10$. Indeed, numerical approximations suggest that $S_{9|1}(\t)=\eta(\t)^3+3\eta(9\t)^3$ and 
$$
S_{10|1}(\t)=C\left(2\eta(\t)^3+7\frac{\eta(2\t)\eta(10\t)^3}{\eta(5\t)}\right)
$$
for some $C\in\C$. Observe that the functions $\eta(\t)^3/\eta(9\t)^3$ and $\eta(\t)^3\eta(5\t)/\eta(2\t)\eta(10\t)^3$ are hauptmoduln for $\G_0(9)$ and $\G_0(10)$, respectively, and so $S_{3/2,\e^3}(\G_0(n))$ is not one dimensional for $n=9$ or $n=10$. Indeed, there are no $M_{24}$ classes corresponding to $(n,h)=(9,1)$ or $(10,1)$.

Also in the case that $h>1$, it is extremely non-trivial that for all the pairs $(n,h)$ arising from $M_{24}$, just the correct combinations of genus, cusp widths and multipliers conspire to force the vanishing of $S_{n|h}$.

 \end{rmk}

\subsection{Determining $R$}
\label{RRR}

First we will establish the identity $H_g = -2R_{n|1}$, for $n=n_g$ and for those $g$ with $h_g=1$. Recall from the discussion in \S \ref{M24facts} that these are the $M_{24}$ classes whose action on the natural permutation representation has at least one fixed point (the $M_{23}$ classes). 
To do this, we consider the function 
\be\label{def_G}
G_{n|1} = \h^3 \Big( \chi(g) \,R_{n|1} - \l_{n} H_g\Big)
\ee 
with $g$ being an $M_{24}$ class with $n_g|h_g=n|1$. By construction, it has $G_{n|1} =( \chi(g) + 2 \l_n) +{\cal O}(q)$ near the infinite cusp.

In \S\ref{Mock} we saw that $H_g$ is a weight $1/2$ mock modular form for $\G_0(n)$ with shadow $\chi(g)\h^3$ and the same multiplier as $\h^{-3}$.
On the other hand, in \S  \ref{sec:Var} we have proven that the Rademacher sums $R_{n|1}$ are also weight $1/2$ mock modular forms for $\G_0(n)$ with the same multiplier. Moreover, it has as its shadow $S_{n|1}$, which was proven to be given by $\l_n \h^3$ for some $\l_n \in \C$ in the previous subsection. Taken together, we see that the function $G_{n|1}: \H^* \to \PP^1$ is a weight 2 modular form of $\G_0(n)$ with trivial multiplier.

Using the properties of such weight 2 modular forms, we will be able to show that $G_{n|1}=0$ for the values of $n$ of interest to us. 
In particular, we have $\chi(g) + 2 \l_n=0$ as the constant term in the expansion near the infinite cusp. Since $\chi(g)\neq 0 $ when $h_g=1$, this implies  
$H_g = -2R_{n|1}$ for  $g \in M_{24}$ with $h_g=1$.

\begin{lem}\label{Id_R_1A}
Let $g$ be the identity element of $M_{24}$ and set $n=n_g=1$ and $h=h_g=1$. Then $H=H_g=-2R_{n|h}$.
\end{lem}
\begin{proof}
From the above discussion we know the function 
$$
G_{n|1}=\eta^3\,(24R_{n|1}-\l_{n}H_g)
$$ 
is a modular form of weight $2$ for $\SL_2(\Z)$ with trivial multiplier. There are no non-zero modular forms of weight $2$ for $\SL_2(\Z)$ so $G_{n|1}$ vanishes identically and by the above argument this proves the lemma.
\end{proof}
 
 To prove $G_{n|1}=0$ for $n>1$ we need to study its behaviour near the other cusp representatives of $\G_0(n)$. 
For a cusp $\k=\s \inf $, $\s\in \SL_2(\Z)$, with width $v$ in $\G_0(n)$, we shall again consider its transformation under a scaling element $\til\s$ (cf. \eq{scaling})
\[
G_{n|1}^{(\k)} (\t) =  G_{n|1}\big\lvert_{1,2} \s\,(v( \t+\b) \hspace{.02mm}) \;. 
\]
To be more precise, we would like to know the form of the Fourier expansion of $G_{n|1}^{(\k)} (\t) $. As will be shown explicitly in Appendix \ref{The Rademacher Sums at Other Cusps}, we have 
\[
G_{n|1}^{(\k)} (\t) = {\cal O}(q^{v/8+\D})
\]
as $\t\to  i\inf$ with $\D\geq 0$, which corresponds to the fact that both $R_{n|1}$ and $H_g$ have no pole at any cusp other than  the infinite cusp. 
By construction, $G_{n|1}^{(\k)} $ is invariant under $\t \to \t+1$ and hence the above argument shows 
\be\label{expansion_G}
G_{n|1}^{(\k)} (\t) = {\cal O}(q^{\lceil v/8\rceil})\;,
\ee
where the ceiling function $\lceil x \rceil$ gives the smallest integer not less than $x\in \R$.

\begin{lem}\label{lem:h1np}
Suppose that $g\in M_{24}$ has prime order $n=n_g$ and $h=h_g=1$. Then $H_g=-2R_{n|h}$.
\end{lem}
\begin{proof}
The primes dividing the order of $M_{24}$ are exactly those primes $p$ for which $(p+1)|24$ so $n\in\{2,3,5,7,11,23\}$. 
For these $n$, the only non-infinite cusp of $\G_0(n)$ is the one represented by 0, with width $n$. 
Take the Fricke involution $w_n$ to be the scaling element.  
From \eq{expansion_G} we conclude that $G_{n|1}^{(\k=0)}$ is ${\cal O}(q)$ for $n\in\{2,3,5,7\}$ and ${\cal O}(q^2)$ for $n=11$ and ${\cal O}(q^3)$ for $n=23$.

Since the Fricke involution $w_n$ for $\G_0(n)$ normalises $\G_0(n)$,  $G_{n|1}^{(\k=0)}$  also  belongs to $M_2(\G_0(n))$. 
The space $M_2(\G_0(n))$ of such weight 2 modular forms contains the space $S_2(\G_0(n))$ of  weight 2 cusp forms. The dimension of the latter is given by  the genus of the modular curve $X_0(n)$. A complement in the former to the latter is spanned by $\{\L_m,\,  m>1,\,m|n\}$ with
\begin{gather}\label{eqn:psin}
\begin{split}
\L_m(\t)&=m \, q \frac{d}{{d}q}\log(\eta(n\t)/\eta(\t))=
\frac{m}{24} \Big( m \,E_2(m\t) - E_2(\t) \Big)\\ 	
&=\frac{m(m-1)}{24}+\sum_{k= 1}^\inf m \s(k)(mq^{mk}-q^k)
\end{split}
\end{gather}
where  $\s(k)$ is the sum of the divisors of $k$ and $E_2$ is the holomorphic quasi-modular Eisenstein series of weight 2. 

When  $n$ is prime, obviously the complement in $M_2(\G_0(n))$ to  $S_2(\G_0(n))$ is one-dimensional and spanned by $\L_n$.
Since the genus of $X_0(n)$ is zero for $n\in\{2,3,5,7\}$, the space $M_2(\G_0(n))$ is one-dimensional and spanned by $\L_n$.
Now, with $\L_n$ having non-vanishing constant term, the statement $G_{n|1}^{(\k=0)}(\t)={\cal O}(q)$ implies that $G_{n|1}^{(\k=0)}(\t)$ vanishes identically. 

If $n=11$ then $M_2(\G_0(n))$ is two dimensional, spanned by $\L_n$ and the cusp form $\eta_g(\t)=\eta(\t)^2\eta(11\t)^2$, and since $\L_n$ and $\eta_g$ have non-vanishing coefficients of $q^0$ and $q^1$, respectively, we conclude that $G_{n|1}^{(\k=0)}(\t)={\cal O}(q^2)$ implies the vanishing of $G_{n|1}^{(\k=0)}$ in this case also. When $n=23$, the space $M_2(\G_0(n))$ is spanned by $\L_n$, a cusp form $\f_{23,1}=\eta_g(\t)^2=\eta(\t)^2\eta(23\t)^2$ and a further cusp form 
\be\label{phi232}
	\phi_{23,2}(\t)=
	\frac{\eta(\tau)^3\eta(23\tau)^3}{\eta(2\tau)\eta(46\tau)}
	+4\eta(\tau)\eta(2\tau)\eta(23\tau)\eta(46\tau)
	+4\eta(2\tau)^2\eta(46\tau)^2.
\ee
Inspecting the Fourier expansions of $\L_n$, $\f_{23,1}$ and $\f_{23,2}$ we see that there is no non-zero linear combination that is ${\cal O}(q^3)$. We conclude that $G_{n|1}^{(\k=0)}(\t)$ vanishes identically for all the values of $n$ in question. So $G_{n|1}$ also vanishes and the desired result follows.
\end{proof}

\begin{lem}\label{lem:h1npq}
Suppose that $g\in M_{24}$ has order $n=n_g$ a product of two distinct primes and $h=h_g=1$. Then $H_g=-2R_{n|h}$.
\end{lem}
\begin{proof}
According to Table \ref{examples_eta} the values of $n$ in question are $6$, $14$ and $15$. 
From the discussion before \eq{eqn:psin} we see that the complement to $S_2(\G_0(n))$ in $M_2(\G_0(n))$ is three-dimensional and spanned by the weight two modular forms $\L_{e_1}, \L_{e_2} $ and $\L_{e_3}$, where $e_1<e_2<e_3=n$ are the three (exact) divisors of $n$ that are larger than 1. 
Therefore, we have 
$$
	G_{n|1}=\alpha_1 \L_{e_1}
	+\alpha_2 \L_{e_2}
	+\alpha_3 \L_{e_3}
	+\f
$$
with $\f \in S_2(\G_0(n))$. In other words, $\f$ vanishes when $n=6$ and is a multiple of $\h_g$ when $n=14,15$, in accordance with the fact that the genera of $X_0(6)$, $X_0(14)$, $X_0(15)$ are 0, 1, 1, respectively. 

For each  $e_i$ there is a corresponding cusp of $\G_0(n)$ with width ${e_i}$, which is represented by $\k_i=\tfrac{e_i}{n}$. These are the only cusps of $\G_0(n)$ apart from the infinite cusp, as listed in Table \ref{Cusps of G}.
As discussed before (cf. \eq{expansion_G}), the corresponding function $G_{n|1}^{(\k_i)} (\t) $ has the asymptotic behaviour 
$$G_{n|1}^{(\k_i)} (\t) = {\cal O}(q^{\D+\tfrac{e_i}{8}})\;.$$

Consider a scaling element which is an  Atkin-Lehner involution  $w_{e_i}$ with $\k_i = w_{e_i} \inf$ (note that it is different from the Fricke involution when $e_i \neq n$). With help from the identity
\be\label{old_form_id}
\tfrac{1}{e}\L_e\lvert_{1,2} w_f = \tfrac{1}{e\ast f}\L_{e\ast f} -\tfrac{1}{f}\L_f 
\ee
where $e*f=ef/(e,f)^2$, we obtain three equations on $\a_i$ from the vanishing of the constant terms in $G_{n|1}^{(\k_i)} (\t)$. 
Solving the resulting linear system we quickly deduce that $\a_1=\a_2=\a_3=0$ and hence $G_{n|1}(\t)$ is given by the cusp form $\f$ of $\G_0(n)$. 

Since $X_0(6)$ has genus zero we conclude that $G_{n|1}(\t)=0$ for $n=6$. 
For $n=14,15$, to prove that $G_{n|1}(\t)=\f=0$ let us focus on the 0 cusp with width $n$.  
Take the scaling element to be the Fricke involution $w_{n}$. From the fact that the cusp form $\h_g$ is an eigenfunction (with eigenvalue $-1$) of $w_n$ we get $G^{(\k)}_{n|1}(\t)= C \h_g =- C q + {\cal O}(q^2)$ for some $C\in \C$. But  from \eq{expansion_G} we see that $G^{(\k)}_{n|1}(\t)= {\cal O}(q^2)$ for $v=n=14,15$. Hence we conclude $G_{n|1}=C=0$. \end{proof}

 \begin{lem}\label{lem:h1n48}
Suppose that $g\in M_{24}$ is such that $h=h_g=1$ and $n=n_g$ is $4$ or $8$. Then $H_g=-2R_{n|h}$.
\end{lem}
\begin{proof}
If $n=4$ then $\dim M_2(\G_0(n))=2$ and $M_2(\G_0(n))$ is spanned by the modular forms $\L_{2}$ and $\L_{4}$ (cf. (\ref{eqn:psin})). If $n=8$ then $\dim M_2(\G_0(n))=3$ and $M_2(\G_0(n))$ is spanned by $\L_2$, $\L_4$ and $\L_8$. In both cases $\dim S_2(\G_0(n))=0$.

Consider the case that $n=4$, where we have $G_{n|1} = \a \L_2 + \b \L_4$. 
There are two non-infinite cusps: one represented by 0 with width 4 and the other represented by $1/2$ with width 1. 
Taking as their scaling elements $w_4$ and $T^{1/2} w_4$ respectively, we get two equations on the coefficients $\a,\b$ from the requirement that both 
$G_{n|1}^{(\k=0)}$ and $G_{n|1}^{(\k=1/2)}$ are $ {\cal O}(q)$, which force both $\a$ and $\b$ and hence $G_{n|1}$ to vanish. 
In arriving at these equations, we have used the identity \eq{old_form_id} and 
\[
E_2(\t+1/2)=-E_2(\t)+6E_2(2\t)-4E_2(4\t)\;. 
\]

 The argument for $n=8$ is very similar. We have $F=\a\L_2+\b\L_4+\g\L_8$ for some $\a,\b,\g\in\C$. We may take $0$, $1/2$ and $1/4$ as representatives for the three non-infinite cusps of $\G_0(8)$. We may take $w_8$, $T^{1/2}w_8$ and $w_8T^{1/2}w_8$, respectively, as scaling elements for these cusp representatives. Using the identities given above we compute expressions for the constant term at each non-infinite cusp as linear equations in $\a$, $\b$ and $\g$. From the vanishing of each constant term we deduce that $\a=\b=\g=0$ and hence $G_{n|1}=0$ also. Then the required identity $H_g=-2R_{n|h}$ follows as before.

\end{proof}

 The above Lemmas \ref{Id_R_1A}-\ref{lem:h1n48}, when taken together, show that $H_g = -2R_{n_g|1}$ for  $g \in M_{24}$ with $h_g=1$ 
 Now we will continue to show  the identity  for the remaining $M_{24}$ classes with $n_g=n$, $h_g=h>1$.

 \begin{lem}
Let $g$ be an element of $M_{24}$ that has $h_g>1$ and $n_g\neq 21$. Then $H_g=-2R_{n|h}$.
\end{lem}
\begin{proof}

To prove the identity $-2R_{n|h} = H_g$ for the remaining classes $[g]$ with $h_g>1$, we consider the following function 
$$
K_{n|h}  =  \frac{R_{n|h}}{H_g} \;. 
$$
Since both $R_{n|h}$ and $H_g=-\til T_g/ \h^3$ are known to be weight $1/2$ modular forms which transform with a multiplier $\r_{n|h}$ on $\G_0(n)$, we conclude that $K_{n|h}$ is a modular function on  $\G_0(n)$. As in the previous subsection, such a modular function has to be either constant or have a pole. 
For $g \in M_{24}$ with $h_g>1$ and $n_g\neq 21$, the McKay--Thompson series $H_g$ are all given by certain $\h$-quotients (cf. Table \ref{h_g}). Since $\h$ does not have a zero in $\H$, we only have to check that  $K_{n|h}$ has no pole at any cusp of $\G_0(n)$. 
 At the infinite cusp we have $R_{n|h} = q^{-1/8} + {\cal O}(q^{7/8})$ and $H_g = -2 q^{-1/8} + {\cal O}(q^{7/8})$ and hence $K_{n|h}$ has no pole at the infinite cusp.

To see that it is also bounded near a non-infinite cusp $\k = \s \inf =\til \s \inf$, $\s \in \SL_2(\Z)$, with a scaling element $\til\s$, we consider the function 
$$
G^{(\k)}_{n|h} =G^{(\k)}_{n|h} \big(\s \,(v(\t+\b))\big) =    \frac{R_{n|h}^{(\k)}(\t)}{\til T_g/ \h^3 \vert_{\e^{-3},1/2} \,\s(v(\t+\b))}
$$
where $R_{n|h}^{(\k)} (\t) $ has been shown in Appendix  \ref{The Rademacher Sums at Other Cusps} to have an expansion 
\be\label{expand_R_other_cusp}
R_{n|h}^{(\k)} (\t) = R_{n|h}^{(\k)} \big\lvert_{\e^3,1/2} \s\, (v(\t+\b))  = \sum_{k=1}^\inf c_{\G,\r}^{(\k)} (k-\n) \ex(\t(k-\n)\hspace{.2mm})\;,
\ee
where $\nu=\{\!\{ \tfrac{m}{h}+\tfrac{v}{8}\}\!\}$ as explained in \eq{S_nh_other_cusp}.
With direct calculation, one can check that also 
$$
\frac{\til T_g}{ \h^3}\big \vert_{\e^{-3},1/2} \,\s(v(\t+\b)\hspace{.2mm}) =a_g \,q^{1-\n} +  {\cal O}( q^{2-\n} ) 
$$
at every non-infinite cusp $\k$ of $\G_0(n)$ with some $a_g \in \C$, $a_g \neq 0$. 

This shows that $K_{n|h}$ is a modular function with no pole and hence constant, which is $-1/2$ from its value at the infinite cusp.

\end{proof}

\begin{lem}
Suppose that $g\in M_{24}$ is such that $n=n_g=21$ and $h=h_g=3$. Then $H_g=-2R_{n|h}$.
\end{lem}
\begin{proof}
Set $\f_{21,1}(\t)=\eta(7\t)^3/\eta(3\t)\eta(21\t)$ and $\f_{21,2} (\t)=\eta(\t)^3/\eta(3\t)^2$. Then both $\f_{21,1}$ and $\f_{21,2}$ are (meromorphic) modular forms of weight $1/2$ for $\G_0(21)$ with multiplier system $\rho\e^{-3}$ where $\rho=\rho_{21|3}$, so that 
$$
G=\frac{A R_{n|h}+B \f_{21,1}}{\f_{21,2}}\quad, \quad A,B \in \C
$$
is a modular function for $\G_0(21)$ whose only poles are at the non-infinite cusps of $\G_0(21)$. 
Since $R_{n|h}, \f_{21,2} ,\f_{21,2}$ all have the expansion $q^{-1/8} + {\cal O}(q^{7/8})$ near the infinite cusp, we see that $G$ has no pole at the infinite cusp. 
For $e_{1,2,3}$ with $e_1<e_2<e_3=21$ the exact divisors of $21$ that are larger than 1 let $w_{e_i}$ be an element of the Atkin--Lehner involution of $\G_0(21)$ associated to $e_i$. Then the elements $w_{e_i}$ furnish scaling elements at the respective representatives $\k_i=w_{e_i}\inf$ for the three non-infinite cusps of $\G_0(21)$, represented by $\tfrac{e_i}{n}$ having width $e_i$.  Using the fact that the slash operator $|_{1/2,\e}w_{e_i}$ maps $\eta(e_j\t)$ to  $\eta(e_i* e_j \t)$ (up to a non-vanishing scale factor), we see that the function $G^{(\k_i)}=G(w_{e_i}\t)$ remains bounded as $\t\to i\inf$ except possibly in the case that $e_3=21$, in which case we have
$$
G^{(\k=0)}(\t) =G(w_{21}\t)=\left(A' R^{(\k=0)}_{n|h}(\t) +B'\, \frac{\eta(3\t)^3}{\eta(7\t)\eta(\t)}\right)\frac{\eta(7\t)^2}{\eta(21\t)^3}\;,\;A',B' \in \C.
$$
From \eq{expand_R_other_cusp} we see that $G^{(\k=0)}(\t) $ has a Fourier expansion of the form $C q^{-2} +{\cal O}(q^{-1})$. 
Setting $D=2x$ where $x$ is the image of $w_{21}\inf=0$ under the natural map $\varphi_{21}:\H^*\to X=X_0(21)$ we see that $f$ must belong to the space $K_X(D)$ of meromorphic functions on $X=X_0(21)$ having a pole of order at most $2$ at $x$ and no other poles. The genus of $X_0(21)$ is $1$ so according to the Riemann--Roch theorem (cf. (\ref{eqn:RR})) we have $\dim K_X(D)=\dim\Omega_X(D)+2$ where $\Omega_X(D)$ denotes the space of holomorphic differentials on $X$ vanishing to order at least $2$ at $x$. The full space $\Omega_X$ of holomorphic differentials on $X$ is in correspondence with $S_2(\G_0(21))$, and this one-dimensional space is in turn spanned by the $L$-series for the elliptic curve $y^2+xy=x^3+x$. This cusp form has a Fourier expansion of the form $q+{\cal O}(q^2)$ and is an eigenform for $w_{21}$ so we conclude that it does not belong to $\Omega_X(D)$, and thus $\dim \Omega_X(D)=0$ and $\dim K_X(D)=2$. Now $K_X(D)$ includes constant functions on $X$ and it also includes the non-constant function $\f_{21,1}/\f_{21,2}$, so it is spanned by these functions, and we may conclude that $R_{n|h}=C_1\f_{21,1}+C_2\f_{21,2}$ for some $C_1,C_2\in\C$. By comparison of polar terms in the Fourier expansions of $R_{n|h}$, $\f_{21,1}$ and $\f_{21,2}$ we have $C_1+C_2=1$. Now consider the expansion of $R_{n|h}$ at the cusp represented by $w_7\inf$ (which is also represented by $1/3$). We have $(\f_{21,1}|w_7)(\t)=D_1q^{-21/24}+{\cal O}(q^{3/24})$ and $(\f_{21,2}|w_7)(\t)=D_2q^{-21/24}+{\cal O}(q^{3/24})$ and the fact that $R_{n|h}$ has no poles away from the infinite cusp implies that $C_1+7C_2=0$. We conclude that $6R_{n|h}=7\f_{21,1}-\f_{21,2}$. From Table \ref{h_g} we have the explicit expression $3H_g=\f_{21,2}-7\f_{21,1}$. This proves the lemma.
\end{proof}

\section{Spectral Theory}
\label{Spec}
\setcounter{equation}{0}

In this section we would like to demonstrate the convergence of the Selberg--Kloosterman zeta function $Z_{0,k}(s,\r\e^{\mathsmaller{-3}})$, defined in \eq{defn:genKlosum} and \eq{defn:SelKlozeta}, of $\G=\G_0(n)$ for $n$ a positive integer. As before, we take $\psi =\e^{\mathsmaller{-3}} \r_{n\lvert h}$ for some $h$ dividing both $24$ and $n$  (cf. (\ref{rho_multiplier})). Combined with the discussions in \S \ref{sec:Conv} and \S \ref{sec:Coeffs}, Theorem \ref{thm:SelKloZetaConv} completes the proof of the convergence of the Rademacher sum \eq{defn:RGrho}.

Let $w$ be a positive real number and let $\psi$ be a multiplier system of weight $w$ for $\G$. We say  a function $g(z)$ on $\H$ is {\em automorphic for $\G$ of weight $w$ with respect to $\psi$} if
$$
	\psi(\g)g(\g z)\exp(iw\arg({\jac(\g,z)})/2)=g(z)
$$ 
for all $\g\in \G$, where $\arg(\cdot)$ is defined so that $-\pi<\arg(z)\leq \pi$ for $z\neq 0$, and write $\mathcal{H}_{\psi,w}$ for the Hilbert space consisting of those automorphic functions which also satisfy the growth condition 
$$
	\int\!\!\!\int_{\G\backslash\H}|g(\t)|^2\frac{{\rm d}x{\rm d}y}{y^2}<\inf
$$
where $x=\Re(\t)$ and $y=\Im(\t)$ for $z\in \H$. We consider the operator $\Delta_w$ on $\mathcal{H}_{\psi,w}$ given by
$$
	\Delta_{w}=y^2\left(\frac{\partial^2}{\partial x^2}+\frac{\partial^2}{\partial y^2}\right)-iwy\frac{\partial}{\partial x}.
$$
According to \cite{Sel_EstFouCoeffs} if $\psi$ is non-trivial on every parabolic subgroup of $\G$ then $\Delta_{w}$ has a discrete spectrum of eigenfunctions in $\mathcal{H}_{\psi,w}$ and there exists a complete orthonormal system $\{u_j(z)\}$ of eigenfunctions satisfying
$$
	\Delta_w u_j(z)+\lambda_j u_j(z)=0
$$
for a set $\{\lambda_j\}$ of eigenvalues which are all real. (Roelcke proves in \cite{Roe_EigPrbAutFrmsHypPln} that $\Delta_{w}$ has a unique self-adjoint extension to $\mathcal{H}_{\psi,w}$.) Further, the $\lambda_j$ are all positive except possibly for a finite number which are of the form 
\be\label{eqn:cuspformeigenval}
(w/2-l)(1+l-w/2)
\ee
where $l$ is a non-negative integer less than $w/2$. If there is an eigenvalue of the form (\ref{eqn:cuspformeigenval}) then, negative or otherwise, it's multiplicity is equal to the dimension of the space $S_{w-2l}(\G,\psi)$ of (holomorphic) cusp forms for $\G$ of weight $w-2l$ with respect to $\psi$.

By the trivial estimate $|S(m,k,c,\psi)|<c$ the series (\ref{defn:SelKlozeta}) defining the Selberg--Kloosterman zeta function $Z_{m,k}(s,\psi)$ converges absolutely and locally uniformly for $\Re(s)>1$ and thus defines a holomorphic function in this domain. Selberg demonstrates in \cite{Sel_EstFouCoeffs} that this function admits an analytic continuation to a function meromorphic in the entire $s$-plane that is holomorphic for $\Re(s)>1/2$ except possibly for a finite number of simple poles lying on the real segment $1/2<s< 1$. Poles on this segment can only occur at points of the form $1/2+\sqrt{1/4-\lambda_j}$ where $0<\lambda_j<1/4$ and $\lambda_j$ is an eigenvalue of $\Delta_w$, and such a pole does occur if $\lambda_j$ is not of the form (\ref{eqn:cuspformeigenval}) and the coefficients of $\ex((m-\alpha)x)$ and $\ex((n-\alpha)x)$ in the Fourier expansion (with respect to $x$) of the corresponding eigenfunction $u_j(z)$ are both non-zero. If $\lambda_j$ is of the form (\ref{eqn:cuspformeigenval}) with $l=\lfloor w/2\rfloor$ then there is a pole at $s=1+\lfloor w/2\rfloor-k/2$ just in the case that there is a cusp form $f(\t)\in S_{w-2\lfloor w/2\rfloor}(\G,\psi)$ for which the coefficients of $q^{m-\alpha}$ and $q^{n-\alpha}$ (in its Fourier expansion with respect to $\t$) are both non-zero.

For the case of relevance for us we have $w=1/2$ and so all the $\lambda_j$ will be positive. Roelcke proves in \cite{Roe_EigPrbAutFrmsHypPln} that the minimal value of a $\lambda_j$ is $3/16$ in case $w=1/2$ (see \cite{Sar_ANT} for a nice exposition) so we conclude that there are no poles for $Z_{0,k}(s,\rho\e^{\mathsmaller{-3}})$ in $\Re(s)>3/4$. The only possible value for $l$ in (\ref{eqn:cuspformeigenval}) is $l=0$ so that the occurrence of a pole at $s=1+\lfloor w/2\rfloor-w/2=3/4$ depends upon the existence of a cusp form $f\in S_{w}(\G,\psi)$ for which the coefficients of $q^{0-\alpha}$ and $q^{k-\alpha}$ are both non-zero. We have $-\alpha=-1/8$ so there is no such cusp form and we conclude that $Z_{0,k}(s,\rho\e^{\mathsmaller{-3}})$ is in fact analytic in $\Re(s)>3/4-\delta$ for some $\delta>0$ and has finitely many simple poles lying on the line segment $1/2<s<3/4$.

It remains to show that the value of the analytically continued function $Z_{0,k}(s,\rho\e^{\mathsmaller{-3}})$ at $s=3/4$ is indeed the sum
$$
	\sum_{c>0}\frac{S(0,k,c,\rho\e^{\mathsmaller{-3}})}{c^{3/2}}
$$
that we expect it is. To establish this we adopt the approach presented in \S2 of \cite{Kno_SmlPosPowTheta}\footnote{If you consult this reference please see also \S II of \cite{Kno_SmlPosWgt} for some important corrections.} which is an adaptation of a technique commonly used in proving the prime number theorem (cf. \cite{Dav_MNT}). (For a very crisp application of what is essentially the same argument, see the proof of Theorem 2 in \cite{GolSar_Kloo}.) Fix a group $\G$, a multiplier $\psi=\r\e^{\mathsmaller{-3}}$, and an integer $k$. We will write $Z(s)=Z_{0,k}(s,\rho\e^{\mathsmaller{-3}})$ to ease notation. Define
$$
\Sigma(x)=\sum_{0<c<x} \frac{S(0,k,c,\rho\e^{\mathsmaller{-3}})}{c^{3/2}}
$$
for $x>0$. We require to show that $\lim_{x\to \infty}\Sigma(x)=Z(3/4)$. Observe that for $a>0$ and $x\neq c$ the integral 
$$
\frac{1}{\tpi}\int_{a-i\inf}^{a+i\inf}\left(\frac{x}{c}\right)^t\frac{{\rm d}t}{t}
$$
is $1$ or $0$ according as $x>c$ or $x<c$. Consequently we may rewrite $\Sigma(x)$ as 
$$
\Sigma(x)=\sum_{c>0}\frac{S(0,k,c,\rho\e^{\mathsmaller{-3}})}{c^{3/2}}
\frac{1}{\tpi}\int_{a-i\inf}^{a+i\inf}\left(\frac{x}{c}\right)^t\frac{{\rm d}t}{t}
$$
for non-integral values of $x$. Choose $\varepsilon>0$ and set $a=1/2+\varepsilon$ so that $Z(t/2+3/4)$ converges absolutely for $\Re(t)=a$. For $T>0$ define $\Sigma_T(x)$ by setting
$$
\Sigma_T(x)=\frac{1}{\tpi}\int_{a-iT}^{a+iT}Z(t/2+3/4)x^t\frac{{\rm d}t}{t}.
$$
Then for the difference $\Sigma(x)-\Sigma_T(x)$ we have the upper bound
\be\label{eqn:SigDiff}
	|\Sigma(x)-\Sigma_T(x)|<x^a\sum_{c>0}c^{-1-\varepsilon}
	\min\left(1,\frac{1}{T|\log(x/c)|}\right)
\ee
thanks to the following estimate, which may be found in \S17 of \cite{Dav_MNT}.
$$
\frac{1}{2\pi}
\left|
\int_{a-i\inf}^{a+i\inf}\left(\frac{x}{c}\right)^t\frac{{\rm d}t}{t}
-\int_{a-iT}^{a+iT}\left(\frac{x}{c}\right)^t\frac{{\rm d}t}{t}
\right|
<\left(\frac{x}{c}\right)^a\min\left(1,\frac{1}{T|\log(x/c)|}\right)
$$
Following exactly the argument of the last paragraph of p.171 of \cite{Kno_SmlPosPowTheta} we can replace the summation in (\ref{eqn:SigDiff}) with the more explicit bound
\be\label{eqn:DiffSigSigT}
	|\Sigma(x)-\Sigma_T(x)|
	<
	K\frac{x^{1/2}}{T}(x^{\varepsilon}+\log x)
\ee
where $K$ is some constant (that does not depend on our choice of $n$). 

Consider now the integral
$$
	\frac{1}{\tpi}\int_CZ(t/2+3/4)\,x^t\,\frac{{\rm d}t}{t}
$$
where $C=C_0+C_1+C_2+C_3$ is the positively oriented rectangle with corners $(\pm 1/2)+\varepsilon\pm iT$. For concreteness let us say that $C_0$ is the right most vertical component  and $C_1$ is the upper most horizontal component, \&c. Then the contribution of $C_0$ is exactly $\Sigma_T(x)$. We denote the contribution of $C_i$ by $I_i$ for $i\neq 0$. By the residue theorem we have
\be\label{eqn:ResThmSigT}
	\Sigma_T(x)+I_1+I_2+I_3=Z(3/4)+\sum_{j\in J}	{\varrho_j}
	x^{2s_j-3/2}
\ee
where $\{s_j\}_{j\in J}$ is an enumeration of the (finite) set of poles of $Z(s)$ in the interval $1/2+\varepsilon/2<s<3/4$ (we may choose $\varepsilon>0$ so that $s=1/2+\varepsilon/2$ is not a pole) and $\varrho_j$ is the residue of $Z(s)$ at $s=s_j$ divided by $2s_j-3/2$.

We wish to determine what happens when $x\to \inf$ in (\ref{eqn:ResThmSigT}). Observe, for a start, that the sum over $j$ vanishes in this limit since $2s_j<3/2$ for all $j\in J$. We claim that the integrals $I_i$ also vanish in this limit. To estimate the $I_i$ we utilise the growth condition
$$
	|Z(s)|=O\left(\frac{
		|s|^{1/2}}{\Re(s)-1/2}\right)
$$
as $|\Im(s)|\to \infty$, valid for $\Re(s)>1/2$. Such a result was first established by Goldfeld--Sarnak in \cite{GolSar_Kloo}. We have applied the generalisation found in \cite{Pri_GnlzdGolSarEst} due to Pribitkin. Since $Z(s)$ converges absolutely for $\Re(s)>1$ we have that $|Z(t/2+3/4)|={\cal O}(1)$ as $|\Im(t)|\to \inf$ when $\Re(t)=1/2+\varepsilon$, and thus an application of the Phragm\'en--Lindel\"of theorem (cf. \cite{HarRie_DirSeries}, Thm. 14) yields
\be\label{eqn:PLbound}
	|Z(t/2+3/4)|< K_{\varepsilon} \,|\Im(t)|^{1/4-\Re(t)/2+\varepsilon/2}
\ee 
for all $t$ such that $-1/2+\varepsilon \leq \Re(t)\leq 1/2+\varepsilon$ and $|\Im(t)|$ is sufficiently large, for some constant $K_{\varepsilon}$ depending only on $\G$, $\psi=\rho\e^{\mathsmaller{-3}}$, $k$ and $\varepsilon$. Applying (\ref{eqn:PLbound}) to the integrals $I_i$ we obtain that $I_2={\cal O}(T^{1/2}x^{-1/2+\varepsilon})$, and both $I_1$ and $I_3$ are ${\cal O}(T^{-1}x^{1/2+\varepsilon})$ so long as $T<x^2$. Setting $T=x^{2/3}$, letting $x\to \inf$ in (\ref{eqn:ResThmSigT}) and applying (\ref{eqn:DiffSigSigT}) gives us the desired result: $\lim_{x\to \inf}\Sigma(x)=Z(3/4)$. In particular, we have the following theorem:
\begin{thm}\label{thm:SelKloZetaConv}
Let $\G=\G_0(n)$ for some positive integer $n$, let $h$ be a divisor of $\gcd(n,24)$ and set $\rho=\rho_{n|h}$. Then for any positive integer $n$ the Selberg--Kloosterman zeta function $Z_{0,k}(s,\rho\e^{\mathsmaller{-3}})$ converges at $s=3/4$.
\end{thm}

\section{Conclusion and Discussion}

To summarise, motivated by the AdS/CFT correspondence in physics, we propose to replace the genus zero property by the more general property of Rademacher summability as the organising principle of the modular properties of moonshine phenomena. 
In monstrous moonshine, where the McKay--Thompson series are modular functions, the latter coincides with the former and is hence automatically satisfied, as shown in \cite{Duncan2009}. 
The two conditions decouple in the case of the recently conjectured relation between mock modular forms and $M_{24}$:
on the one hand, it is neither necessary nor sufficient that a group $\Gamma$ have genus zero for it to arise as $\Gamma_g$ for some g in $M_{24}$, and on the other hand, we verify the Rademacher summability property in Theorem \ref{thm:HisR}. 
In view of the above, the Rademacher summability property may replace the genus zero property, and applies to both monstrous moonshine and $M_{24}$. This also suggests that AdS/CFT considerations will lead to further elucidation of the moonshine phenomena, and we expect the powerful Rademacher machinery to prove useful in further study of the connection between modular objects and finite groups in general.

This result raises many interesting physical questions. 
In particular, what do we know about the gravity duals, if they exist, of 2d CFTs with sporadic symmetries? 
Let us first concentrate on the sigma models with target spaces involving $K3$ surfaces, as they are relevant for the structures studied in the present paper. Recall that one of the best known examples of the AdS/CFT correspondence conjectures that the two-dimensional CFT describing the Higgs branch of the system of $Q_1$ D1-branes bound to $Q_5$ D5-branes wrapping a $K3$ surface is dual to the type IIB string theory in the background $AdS_3\times S^3\times K3({Q})$
\cite{MaldacenaAdv.Theor.Math.Phys.2:231-2521998}. Here the radius of curvature of $AdS_3$ and $S^3$ is given by $(Q_1Q_5)^{1/4}$ in the six dimensional unit, and $K3(Q)$ is the ``{attractor K3}" whose moduli are partially fixed by the requirement of minimising the BPS mass of the given charges, or equivalently, by the so-called attractor mechanism of supergravity.  Note that not all $K3$ surfaces can be such an attractor $K3$ for a given set of charges. On the conformal field theory side, the Higgs branch is described by a 2d ${\cal N}=(4,4)$ SCFT with central charge 
 $c=6k, \;k=Q_1Q_5+1$, which is a deformation of the sigma model on the symmetric product $S^kK3$, as can be most easily understood by going to the D0-D4 duality frame
 \cite{Vafa1996}. 
 
This AdS/CFT dictionary makes it clear that a ``semi-classical-like" AdS gravity description is only guaranteed to exist when we take the central charge $c=6k$, or equivalently the AdS radius, to be very large. 
For this reason, it is in fact surprising that the Rademacher machinery works so well when applied to the partition functions for the $K3$ sigma model with $k=1$, as demonstrated in the present paper. 
Two comments on this issue are in order here. First, analogous Rademacher formulas also exist for the symmetric product sigma model with higher $k$ and we plan to report on this in more detail in the future \cite{to_appear}. In other words, the Rademacher machinery also works when we have very strong reasons to expect them to work. 
Second, recently there has been important progress in developing the localisation techniques for computing gravity path integrals (see, for example, \cite{Dabholkar2011a}), and we have seen that in some cases the highly quantum path integral actually localises into the same form as if coming from a weakly coupled, semi-classical gravitational theory. For instance, very recently the Rademacher sum for the modular form $1/\eta(\t)^{24}$, which is the partition function of a CFT with central charge as small as $c=24$, has been reproduced from a localised gravity path integral \cite{localisation}. This is very surprising since the corresponding black holes have zero macroscopic entropy (that is, the Bekenstein-Hawking entropy grows only linearly with the charges) and hence a priori we have no reason to expect the gravity path integral to take this ``semi-classical" form. 
These results show that the Rademacher formulas are not really the monopoly of a large radius, semi-classical gravity description, although the former is most easily motivated in the semi-classical limit as we reviewed in \S\ref{Introduction}, 
and suggests that it is not unlikely that a localisation calculation could similarly explain the Rademacher sum formulas derived in the present paper from a gravity point of view. 
In the case of monstrous moonshine, on the other hand, the question of whether a semi-classical theory of gravity dual to the moonshine module conformal field theories, or possibly their higher central charge cousins, is an interesting open puzzle \cite{Witten2007}.  

We end with a few more comments concerning the Rademacher formulas studied in this paper.
Firstly, it is remarkable that the Rademacher sums with such a simple polar term $-2 q^{-1/8}$ compute not only the $M_{24}$-modules underlying the massive states in the $K3$ CFT but also  the twisted Euler characteristics of $K3$ surfaces. 
The latter statement comes from the fact that the Rademacher machinery automatically computes the shadows of the resulting mock modular forms, which are in this case given by the twisted $K3$ Euler characteristics multiplied by the modular form $\h(\t)^3$.
Finally, we would like to mention that none of the twisted elliptic genera which arise from the extra, non-$M_{24}$ symmetries discussed in 
\cite{Gaberdiel2011} lead to (mock) modular forms admitting analogous  Rademacher sum expressions. Another related but logically independent observation is the following. If we were to twist the partition function counting the $1/4$-BPS index of type II theory compactified on $K3\times T^2$ with one of these non-$M_{24}$ symmetries, the resulting candidate 4d twisted partition function would have a system of poles and zeros that is more intricate and complicated than the ones coming from twisting with $M_{24}$ symmetries. This adds to the distinction between the $M_{24}$ symmetries and the extra non-$M_{24}$ discrete symmetries which are present in, for instance, the Gepner models, whose role in moonshine is still obscure.

\section*{Acknowledgements}

We thank Atish Dabholkar, Rajesh Gopakumar, Tom Hartman, \"Ozlem Imamoglu, Jan Manschot, Shiraz Minwalla, Ashoke Sen, Edward Witten, and Don Zagier for many useful discussions, and we thank Amanda Folsom and Wladimir de Azevedo Pribitkin for communication related to convergence. We also thank Alejandra Castro, Thomas Creutzig, Igor Frenkel, Christoph Keller, Alex Maloney, Andy Strominger and S.T. Yau for earlier discussions on related topics. M.C. would like to thank the Institute for Advanced Study for hospitality during part of this work, and the organisers and participants of ``the School and Conference on Modular Forms and mock Modular Forms and their Applications in Arithmetic, Geometry and Physics" for stimulating discussions. The research of M.C. is supported in part by NSF grant DMS-0854971.

\clearpage
\appendix

\section{Dedekind Eta Function}
\label{Dedeta}

The {\em Dedekind eta function}, denoted $\eta(\t)$, is a holomorphic function on the upper half-plane defined by the infinite product 
$$\eta(\t)=q^{1/24}\prod_{n\geq 1}(1-q^n)$$
where $q=\ex(\t)=e^{\tpi \t}$. It is a modular form of weight $1/2$ for the modular group $\SL_2(\Z)$ with multiplier $\e:\SL_2(\Z)\to\C^*$, which means that 
$$\e(\g)\eta(\g\t)\jac(\g,\t)^{1/4}=\eta(\t)$$
for all $\g = \big(\begin{smallmatrix} a&b\\ c&d \end{smallmatrix}\big) \in\SL_2(\Z)$, where $\jac(\g,\t)=(c\t+d)^{-2}$. The {\em multiplier system} $\e$ may be described explicitly as 
\be\label{Dedmult}
\e\bem a&b\\ c&d\eem 
=
\begin{cases}
	\ex(-b/24),&c=0,\,d=1\\
	\ex(-(a+d)/24c+s(d,c)/2+1/8),&c>0
\end{cases}
\ee
where $s(d,c)=\sum_{m=1}^{c-1}(d/c)((md/c))$ and $((x))$ is $0$ for $x\in\Z$ and $x-\lfloor x\rfloor-1/2$ otherwise. We can deduce the values $\e(a,b,c,d)$ for $c<0$, or for $c=0$ and $d=-1$, by observing that $\e(-\g)=\e(\g)\ex(1/4)$ for $\g\in\SL_2(\Z)$.

Let $T$ denote the element of $\SL_2(\Z)$ such that $\tr(T)=2$ and $T\t=\t+1$ for $\t\in \H$. Observe that
$$
\e(T^m\g)=\e(\g T^m)=\ex(-m/24)\e(\g)
$$
for $m\in\Z$.

\section{Bessel Function}
\label{sec:Bessel}
\setcounter{equation}{0}

The {\em Bessel function of the first kind}, denoted $J_{\alpha}(x)$, may be defined by the series expression
\be\label{BesselSeriesJ}
	J_{\alpha}(x)=\sum_{m\geq 0}(-1)^m\frac{(x/2)^{2m+\alpha}}{\G(m+1+\alpha)m!}.
\ee
The function $J_{\alpha}(x)$ is asymptotic to $e^{x}\cos(x-\alpha\pi/4-\pi/4)/\sqrt{2\pi x}$ for $x$ large, and in the special case that $\alpha=1/2$ this fact is borne out by the identity
\be\label{BesselJSin}
	J_{1/2}(x)=\sqrt{\frac{2}{\pi x}}\sin(x).
\ee

The {\em modified Bessel function of the first kind}, denoted $I_{\alpha}(x)$, may be defined by the series expression
\be\label{BesselSeries}
	I_{\alpha}(x)=\sum_{m\geq 0}\frac{(x/2)^{2m+\alpha}}{\G(m+1+\alpha)m!}.
\ee
The function $I_{\alpha}(x)$ is asymptotic to $e^{x}/\sqrt{2\pi x}$ for $x$ large, for any $\alpha$. In the special case that $\alpha=1/2$ this fact is borne out by the identity
\be\label{BesselSinh}
	I_{1/2}(x)=\sqrt{\frac{2}{\pi x}}\sinh(x).
\ee

As $x$ approaches $0$ the functions $J_{1/2}(x)$ and $I_{1/2}(x)$ both tend towards $(x/2)^{\alpha}/\G(1+\alpha)$. Since $\G(3/2)=\sqrt{\pi}/2$ we have
\be\label{BesselIJSmall}
	J_{1/2}(x)\approx \sqrt{\frac{2x}{\pi}},\;I_{1/2}(x)\approx \sqrt{\frac{2x}{\pi}}
\ee
as $x\to 0$ in the special case that $\alpha=1/2$.

\section{Lipschitz Summation}
\label{sec:Lipsum}
\setcounter{equation}{0}

The {\em Lipschitz summation formula} is the identity 
\be\label{eqn:Lipsum}
\frac{(-2\p i)^s}{\G(s)}\sum_{k=1}^\inf {(k-\alpha)^{s-1}}\ex((k-\alpha)\t)=\sum_{\ell\in\Z}\ex(\alpha \ell)(\t+\ell)^{-s},
\ee
valid for $\Re(s)>1$ and $0\leq \alpha <1$, where $\ex(x)=e^{\tpi x}$. A nice proof of this using Poisson summation appears in \cite{KnoRob_RieFnlEqnLipSum}. Observe that both sides of (\ref{eqn:Lipsum}) converge absolutely and uniformly in $\t$ on compact subsets of $\H$. 

For applications to Rademacher sums of weight less than $1$ we require an extension of (\ref{eqn:Lipsum}) to $s=1$. 
Absolute convergence on the right hand side breaks down at this point but we can get by with the following useful analogue.
\begin{lem}\label{lem:LipSumAnlg}
For $0<\alpha<1$ we have 
\be\label{eqn:Lipsum1}
\sum_{k=1}^\inf\ex((k-\alpha)\t)=
\sum_{-K< \ell< K}\ex(\alpha \ell)(-2\pi i)^{-1}(\t+\ell)^{-1}
+E_K(\t)
\ee
where $E_K(\t)={\cal O}(1/K^2)$, locally uniformly for $\t\in\H$.
\end{lem}
\begin{proof}
Our proof follows that of Lemma 4.1 in \cite{Nie_ConstAutInts}. For $\t\in\H$ define a function $f(z)$ by setting
\begin{gather}
	f(z)=\frac{1}{2\pi i z}\frac{\ex(\alpha(z-\t))}{(\ex(z-\t)-1)}.
\end{gather}
Then $f$ has poles at $\t+n$ for each $n\in\Z$ and also has a pole at $z=0$. The residue at $z=0$ is 
$$
	\frac{1}{2\pi i}\frac{\ex(-\alpha\t)}{\ex(-\t)-1}
	=
	\frac{1}{2\pi i}\sum_{k=1}^\inf\ex((k-\alpha)\t),
$$ 
which is $(2\pi i)^{-1}$ times the left hand side of (\ref{eqn:Lipsum1}), and the residue at $z=\t+n$ is $\ex(\alpha n)(2\pi i)^{-2}(\t+n)^{-1}$. Suppose $\t=\s+i t$. Let $K$ be a positive integer and let ${\cal C}$ be the positively oriented boundary of a rectangle with corners $\s\pm (K+1/2)\pm  iL$ where $L>\Im(\t)=t$ and $K+1/2>\Re(\t)=\s$. Then by the residue theorem
$$\int_{{\cal C}}f(z){\rm d}z=\sum_{k=1}^\inf\ex((k-\alpha)\t)-\sum_{-K< \ell< K}\ex(\alpha \ell)(-2\pi i)^{-1}(\t+\ell)^{-1}.$$
On the other hand, the integrals over the horizontal portions of ${\cal C}$ tend to $0$ as $L\to \inf$ since $|f(z)|$ decays exponentially as $\Im(z)\to\inf$. The residues are independent of $L$ for $L$ sufficiently large so taking the limit as $L\to \inf$ we obtain
$$\sum_{k=1}^\inf\ex((k-\alpha)\t)-\sum_{-K< \ell< K}\ex(\alpha \ell)(-2\pi i)^{-1}(\t+\ell)^{-1}=E_K(\t)$$
where $E_K(\t)$ is the limit as $L\to \inf$ of the contributions to $\int_{{\cal C}}f(z){\rm d}z$ coming from the vertical portions of ${\cal C}$.
$$
	E_K(\t)=i\int_{-\inf}^{\inf}(f(\s+K+1/2+iy)-f(\s-K-1/2+iy)){\rm d}y
$$
For convenience set $Q=K+1/2$. After replacing $y$ with $y+t$ we obtain
\be\label{eqn:EK}
	E_K(\t)=-\frac{1}{2\pi}\int_{-\inf}^{\inf}
	\left(
	\frac{\ex(\alpha Q)}{Q+iy+\t}-\frac{\ex(-\alpha Q)}{-Q+iy+\t}
	\right)\frac{\ex(\alpha iy)}{1+\ex(iy)}
	{\rm d}y
\ee
since $\ex(Q)=\ex(-Q)=-1$ for $K$ an integer. Observe that the integrand in (\ref{eqn:EK}) is bounded by $C/Q^2$, for some constant $C$, and decays exponentially as $y\to\pm \inf$ for $0<\alpha<1$. Also, the constant $C$ holds locally uniformly in $\t$. We conclude that $E_K(\t)={\cal O}(1/K^2)$ locally uniformly in $\t$, as required.
\end{proof}

For $\a=0$ we have the identity
\[\lim_{K\to \inf} \sum_{\ell=-K}^K \frac{1}{\t+\ell} = \p \cot \p \t = - i \p \big(1 +  2\sum_{m=1}^\inf q^m \big)\;.
\]

\section{Some Details on Convergence}
\label{details_convergence}
\setcounter{equation}{0}

In this appendix we will give the steps leading to the expression \eq{eqn:LimSum7}, whose convergence we demonstrated in the main text. 
Using the explicit definition \eq{defn:GinfGKcross}, \eq{defn:RS0gamma} and again split the sum over $d$ into a double sum, we obtain
\begin{align}\notag
&\lim_{ K\to \inf}\sum_{\g\in(\G_{\inf}\backslash\G)^{\times}_{<K}}R_{\g}^0(\t) \\\notag
&= 
\ex(\tfrac{1}{8}) \sum_{\substack{0<c< K\\n|c}}c^{-3/2}
\sum_{\substack{-K^2< d< K^2\\(c,d)=1}}
\ex(-\tfrac{\g\inf}{8})\psi(\g)(\t+d/c)^{-1}\\\label{eqn:LimSum3}
&=
\ex(\tfrac{1}{8}) \sum_{\substack{0<c< K\\n|c}}c^{-3/2}
\sum_{\substack{0 \leq d<c\\(c,d)=1}} \ex(-\tfrac{\g\inf}{8})\psi(\g) 
\sum_{\substack{\ell\in\Z\\|d+\ell c|< K^2}}
\frac{\ex(\tfrac{\ell}{8})}{(\t+d/c+\ell)}\;,
\end{align} 
where we have used $ \psi(\g T) = \ex(\tfrac{1}{8})\psi(\g) $ (see Appendix \ref{Dedeta}) in deriving the last line.

Consider  the following two expressions
\be\label{eqn:LimSum4}
\sum_{\substack{\ell\in\Z\\|d+\ell c|< K^2}}
\ex(\tfrac{\ell}{8})(\t+d/c+\ell)^{-1},\;
\sum_{\substack{\ell\in\Z\\|\ell|< K^2/c}}
\ex(\tfrac{\ell}{8})(\t+d/c+\ell)^{-1}\;,
\ee
we will now argue that we can use either of the two in the limit $K\to\inf$. 

The left most of (\ref{eqn:LimSum4}) may include a term where $d+\ell c=x-K^2$ for at most one $0\leq x<d$ and this will not appear in the right hand expression, and the right most of (\ref{eqn:LimSum4}) may include a term where $d+ \ell c=x+K^2$ for at most one $0<x\leq d$ and these will not appear in the left hand expression. We conclude that the difference is ${\cal O}(c/K^2)$ for sufficiently large $K$, uniformly in $\t$, and since $\sum_{0<c< K}c^{1/2}$ is ${\cal O}(K^{3/2})$ we have
$$
\sum_{\substack{0<c< K\\n|c}}
c^{1/2}
\sum_{\substack{0\leq d<c\\(c,d)=1}}
{\cal O}(c/K^2)
={\cal O}(1/K^{1/2})
$$
which vanishes as $K\to \inf$. Therefore we have
\begin{align}\label{eqn:LimSum5}
&\lim_{K\to \inf}\sum_{\mathsmaller{\g\in(\G_{\inf}\backslash\G)^{\times}_{<K}}}R_{\g}^0(\t) =\ex(\tfrac{1}{8}) \lim_{K\to \inf}\sum_{\substack{0<c< K\\n|c}}c^{-3/2}
\sum_{\substack{0 \leq d<c\\(c,d)=1}} \ex(-\tfrac{\g\inf}{8})\psi(\g) 
\sum_{\substack{\ell\in\Z\\|d|< K^2/c}}
\frac{\ex(\tfrac{\ell}{8})}{(\t+d/c+\ell)}\;.
 \end{align} 

Upon applying the Lipschitz summation formula in the form of Lemma \ref{lem:LipSumAnlg} we obtain
\begin{align}\notag
&\lim_{K\to \inf}\sum_{\mathsmaller{\g\in(\G_{\inf}\backslash\G)^{\times}_{<K}}}R_{\g}^0(\t) =2\p \ex(-\tfrac{1}{8})\lim_{K\to \inf}\sum_{\substack{0<c< K\\n|c}}c^{-3/2}
\sum_{\substack{0 \leq d<c\\(c,d)=1}} \ex(-\tfrac{\g\inf}{8})\psi(\g) \\ \notag
& \times \big(-E_{\lfloor K^2/c\rfloor}(\t)+\sum_{k=1}^\inf \ex((k-\tfrac{1}{8})(\t+\tfrac{d}{c}))\,\big)\;.
 \end{align}
For $K$ large the error terms $E_{\lfloor K^2/c\rfloor}(\t)$ are bounded by $E_{K}(\t)$ and according to Lemma \ref{lem:LipSumAnlg} we have $E_{K}={\cal O}(1/K^2)$. It follows that the term involving $-E_{\lfloor K^2/c\rfloor}(\t)$ is ${\cal O}(1/K^{3/2})$ and thus tends to $0$ as $K\to \infty$. Observe that the sum over $k$ in the above formula is absolutely convergent, uniformly in $\tau$ on compact subsets of $\H$. Moving the summation over $k$ past the others and using the explicit expressions for $\psi(\g)$, we readily obtain \eq{eqn:LimSum7}.

\section{Rademacher Sums at Other Cusps}
\label{The Rademacher Sums at Other Cusps}
\setcounter{equation}{0}

To identify $R_{\G,\r}$, we need to study its behaviour at the other cusps of $\G=\G_0(n)$. 
 Focusing on a given cusp representative $\k=\s \inf $, $\s\in \SL_2(\Z)$, with width $v$ in $\G_0(n)$, we need to study  $R_{\G,\r}\big\lvert_{\e^{-3},1/2} \s$.
 For this purpose, let us introduce 
 \begin{align}
 R_\g' &= \e^{\mathsmaller{-3}}(\g) \,e(-\tfrac{\g\t}{8})\,{\rm reg}(\g,\t) \,{\rm jac}^{\mathsmaller{1/4}} (\g,\t)\\
  S_\g' &= \e^{\mathsmaller{3}}(\g) \,e(\tfrac{\g\t}{8})\,{\rm jac}^{\mathsmaller{3/4}} (\g,\t)
 \end{align}
 analogous to the definitions in \S \ref{sec:Conv}, so that 
  \begin{align}\notag
  R_{\G,\r} = \lim_{K\to \inf}\sum_{\g\in(\G_{\inf}\backslash\G)_{<K}}\r(\g)\,R'_{\g}\quad,\quad  S_{\G,\r} =  \lim_{K\to \inf}\sum_{\g\in(\G_{\inf}\backslash\G)_{<K}}\r^{\mathsmaller{-1}}(\g)\,S'_{\g}\;. \end{align}
 A straightforward calculation similar to the one in \S \ref{sec:Var} shows 
  \begin{align}\notag
   R'_{\g}\big\lvert_{\e^{-3},1/2}\s &= R'_{\g\s} + c  J_{\s^{-1}\inf} S'_{\g\s}\quad,\quad c= -\frac{\ex(\tfrac{1}{4})}{\sqrt{8}} \\ 
   S'_{\g}\big\lvert_{\e^{3},3/2}\s &= S'_{\g\s} \;.
  \end{align}
 Using $S_{n|1} = \l_n \h^3$ and $S_{n|h} = 0$ for $h>1$ as shown in \S \ref{SSS}, we get 
 \begin{align}\label{transform_R1}
 R_{n|1}\lvert_{\e^{-3},1/2}\s &=  \lim_{K\to \inf}\sum_{\g\in(\G_{\inf}\backslash\G)_{<K}\s}R'_{\g} + c \l_n  J_{\s^{-1}\inf}\h^3 \\ 
 R_{n|h}\lvert_{\e^{-3},1/2}\s &=  \lim_{K\to \inf}\sum_{\g\in(\G_{\inf}\backslash\G)_{<K}\s}\r(\g\s^{\mathsmaller{-1}})\,R'_{\g} \quad{\rm for}\;\;h>1\;.
 \end{align}
 Splitting the sum into a sum over the double coset space $(\G_{\inf}\backslash\G)_{\mathsmaller{<K}} \s/\G_\inf^v$ and applying the Lipschitz summation formula as before, we obtain that the Fourier expansion of the above sum is 
  \be\label{expand_h_not_1}
 \lim_{K\to \inf}\sum_{\g\in(\G_{\inf}\backslash\G)_{<K}\s}\r(\g\s^{\mathsmaller{-1}})\,R'_{\g}  = \sum_{k=1}^\inf c_k \,\ex((k-\n)\tfrac{\t}{v} )\;.
 \ee
As before we have $0<\n\leq 1$ given by $\n=\{\!\{ \tfrac{m}{h}+\tfrac{v}{8}\}\!\} $, where $\r(\s T^v\s^{-1}) =\ex(\tfrac{m}{h})$ (cf. \eq{S_nh_other_cusp}). 
 
 In particular, for $h=1$ we have 
 \be\label{expand_h_1}
  \lim_{K\to \inf}\sum_{\g\in(\G_{\inf}\backslash\G)_{<K}\s}R'_{\g} =  \sum_{k=1}^\inf c_k \,\ex(\hspace{.2mm}(k-\{\!\{\tfrac{v}{8}\}\!\}\,)\tfrac{\t}{v} )\;.
 \ee

 Using Lemma \ref{Id_R_1A} and equation \eq{h_g_explicit}, we can rewrite the function $G_{n|1}$ as
 \[
 G_{n|1}= \l_n \til T_g + {\chi(g)}\h^3 \big( R_{n|1} + \tfrac{\l_n}{12} R_{1|1}\big) \;.
 \]
From \eq{transform_R1} we deduce that 
\[\big( R_{n|1} + \tfrac{\l_n}{12} R_{1|1}\big)\big\lvert_{\e^{-3},1/2}\s = \lim_{K\to \inf}\sum_{\g\in(\G_{\inf}\backslash\G)_{<K}\s}R'_{\g} + \tfrac{\l_n}{12} 
R_{1|1}
\]
and hence 
\[
 G_{n|1}\big\lvert_{1,2}\,\s = \h^3 \left( -\l_n \left( \frac{\chi(g)}{24} H - \frac{\til T_g\lvert_{1,2}\,\s}{\h^3} \right) + \chi(g)\lim_{K\to \inf}\sum_{\g\in(\G_{\inf}\backslash\G)_{<K}\s}R'_{\g}\right)\;.
\]

One can check from the explicit expressions for $\til T_g$ that $\tfrac{\chi(g)}{24} H - \frac{\til T_g\lvert_{1,2}\,\s}{\h^3}= {\cal O}(1)$, so that the only pole of $H_g$ is at the infinite cusp.  Together with the \eq{expand_h_1} we obtain the expansion \eq{expansion_G}.

\begin{sidewaystable} 
\section{Modular Forms and Representations}
\label{Tables of Fourier Coefficients and Decompositions of Representations}

\centering
 \resizebox{1.08\textwidth}{!}{
 \begin{tabular}[H]{ccccccccccccccccccccccccccc}
 \toprule
$1/\eta_g(\t)$&$1A$&$2A$&$2B$&$3A$&$3B$&$4A$&$4B$&$4C$&$5A$&$6A$&$6B$&$7A$&${7B}$&$8A$&$10A$&$11A$&$12A$&$12B$&$14A$&${14B}$&$15A$&${15B}$&$21A$&${21B}$&$23A$&${23B}$\\\midrule
$q^{0}$&24&8&0&6&0&0&4&0&4&2&0&3&3&2&0&2&0&0&1&1&1&1&0&0&1&1\\
$q^{1}$&324&52&12&27&0&4&16&0&14&7&0&9&9&6&2&5&1&0&3&3&2&2&0&0&2&2\\
$q^{2}$&3200&256&0&104&8&0&48&0&40&16&0&22&22&12&0&10&0&0&4&4&4&4&1&1&3&3\\
$q^{3}$&25650&1122&90&351&0&18&142&6&105&39&0&51&51&28&5&20&3&0&9&9&6&6&0&0&5&5\\
$q^{4}$&176256&4352&0&1080&0&0&368&0&256&80&0&108&108&52&0&36&0&0&12&12&10&10&0&0&7&7\\
$q^{5}$&1073720&15640&520&3107&44&56&928&0&590&175&4&221&221&104&10&65&5&0&23&23&17&17&2&2&11&11\\
$q^{6}$&5930496&52224&0&8424&0&0&2176&0&1296&336&0&432&432&184&0&110&0&0&32&32&24&24&0&0&15&15
\\$q^{7}$&30178575&165087&2535&21762&0&175&4979&27&2740&666&0&819&819&341&20&185&10&0&55&55&37&37&0&0&22&22
\\$q^{8}$&143184000&495872&0&53976&192&0&10864&0&5600&1232&0&1506&1506&580&0&300&0&0&76&76&56&56&3&3&30&30
\\$q^{9}$&639249300&1428612&10908&129141&0&468&23184&0&11130&2289&0&2706&2706&1010&38&481&15&0&122&122&81&81&0&0&42&42\\
                \bottomrule
                &&&& &&&& &&&& &&&&&&&& &&&&&&\\
                &&&& &&&& &&&& &&&&&&&& &&&&&&\\
                &&&& &&&& &&&& &&&&&&&& &&&&&&\\
                \toprule&
               $\chi_{1}$&$\chi_{ 2}$&$\chi_{ 3}$&$\chi_{ 4}$&$\chi_{ 5}$&$\chi_{ 6}$&$\chi_{ 7}$&$\chi_{ 8}$&$\chi_{ 9}$&$\chi_{ 10}$&$\chi_{ 11}$&$\chi_{ 12}$&$\chi_{ 13}$&$\chi_{ 14}$&$\chi_{ 15}$&$\chi_{ 16}$&$\chi_{ 17}$&$\chi_{ 18}$&$\chi_{ 19}$&$\chi_{ 20}$&$\chi_{ 21}$&$\chi_{22}$&$\chi_{ 23}$&$\chi_{ 24}$&$\chi_{ 25}$&$\chi_{ 26} $                \\\midrule
               ${\cal H}_1$ &1&1&0&0&0&0&0&0&0&0&0&0&0&0&0&0&0&0&0&0&0&0&0&0&0&0\\
               ${\cal H}_2$ &3&3&0&0&0&0&1&0&0&0&0&0&0&0&0&0&0&0&0&0&0&0&0&0&0&0\\
               ${\cal H}_3$ &6&8&0&0&0&0&3&2&1&0&0&0&0&0&0&0&1&0&0&0&0&0&0&0&0&0\\
              ${\cal H}_4$  &14&20&0&0&0&0&12&6&4&0&0&0&0&1&0&0&3&0&0&0&1&3&0&0&0&0\\
               ${\cal H}_5$ &27&48&0&0&0&0&33&22&15&0&0&0&0&3&0&0&15&1&0&3&6&16&3&0&0&3\\
                ${\cal H}_6$&59&110&0&0&0&0&97&61&51&0&0&0&0&19&0&0&54&10&9&17&34&69&25&6&9&26\\
               ${\cal H}_7$& 114&249&0&0&6&6&255&174&161&3&3&3&3&70&3&3&190&45&47&88&158&276&147&59&71&194
               \\${\cal H}_8$&235&552&0&0&32&32&687&457&498&40&40&39&39&301&39&39&633&220&269&393&694&1042&758&418&490&1088\\${\cal H}_9$&460&1217&1&1&169&169&1783&1235&1504&255&255&296&296&1126&294&294&2152&994&1252&1730&2850&3870&3528&2354&2656&5544\\${\cal H}_{10}$&924&2677&24&24&731&731&4754&3294&4575&1425&1425&1675&1675&4329&1699&1699&7207&4391&5592&7131&11460&14340&15393&11758&13026&25565\\
                \bottomrule
  \end{tabular}}
  \caption{ \label{eta_prod_decomp}\footnotesize{i) The first ten Fourier coefficients of the McKay--Thompson series $Z_g(\t) = \sum_{k=0}^\inf q^{k-1} (\Tr_{{\cal H}_k} g) =1/\eta_g(\t)$. ii) The decomposition of the first ten $M_{24}$-modules ${\cal H}_k$ into irreducible representations (see the character table  \ref{Character Tables}). } }
 \end{sidewaystable}

 \begin{sidewaystable} \centering
    \resizebox{1.08\textwidth}{!}{
 \begin{tabular}[h!]{c|cccccccccccccccccccccccccc}
 \toprule
{\backslashbox{$\ell$}{$[g]$}}&$1A$&$2A$&$2B$&$3A$&$3B$&$4A$&$4B$&$4C$&$5A$&$6A$&$6B$&$7A$&${7B}$&$8A$&$10A$&$11A$&$12A$&$12B$&$14A$&${14B}$&$15A$&${15B}$&$21A$&${21B}$&$23A$&${23B}$\\\midrule
0&-2&-2&-2&-2&-2&-2&-2&-2&-2&-2&-2&-2&-2&-2&-2&-2&-2&-2&-2&-2&-2&-2&-2&-2&-2&-2\\ 
1&90&-6&10&0&6&-6&2&2&0&0&-2&-1&-1&-2&0&2&0&2&1&1&0&0&-1&-1&-2&-2\\ 2&462&14&-18&-6&0&-2&-2&6&2&2&0&0&0&-2&2&0&-2&0&0&0&-1&-1&0&0&2&2\\ 3&1540&-28&20&10&-14&4&-4&-4&0&2&2&0&0&0&0&0&-2&2&0&0&0&0&0&0&-1&-1\\ 4&4554&42&-38&0&12&-6&2&-6&-6&0&4&4&4&-2&2&0&0&0&0&0&0&0&-2&-2&0&0\\ 5&11592&-56&72&-18&0&-8&8&0&2&-2&0&0&0&0&2&-2&-2&0&0&0&2&2&0&0&0&0\\ 6&27830&86&-90&20&-16&6&-2&6&0&-4&0&-2&-2&2&0&0&0&0&2&2&0&0&-2&-2&0&0\\ 7&61686&-138&118&0&30&6&-10&-2&6&0&-2&2&2&-2&-2&-2&0&-2&2&2&0&0&2&2&0&0\\ 8&131100&188&-180&-30&0&-4&4&-12&0&2&0&-3&-3&0&0&2&2&0&-1&-1&0&0&0&0&0&0\\ 9&265650&-238&258&42&-42&-14&10&10&-10&2&6&0&0&-2&-2&0&-2&-2&0&0&2&2&0&0&0&0\\
       \bottomrule
        \multicolumn{27}{c}{}\\ \multicolumn{27}{c}{}\\ \multicolumn{27}{c}{}\\
          \toprule
  {\backslashbox{$\ell$}{$\chi$}} &
$\chi_{1}$&$\chi_{ 2}$&$\chi_{ 3}$&$\chi_{ 4}$&$\chi_{ 5}$&$\chi_{ 6}$&$\chi_{ 7}$&$\chi_{ 8}$&$\chi_{ 9}$&$\chi_{ 10}$&$\chi_{ 11}$&$\chi_{ 12}$&$\chi_{ 13}$&$\chi_{ 14}$&$\chi_{ 15}$&$\chi_{ 16}$&$\chi_{ 17}$&$\chi_{ 18}$&$\chi_{ 19}$&$\chi_{ 20}$&$\chi_{ 21}$&$\chi_{22}$&$\chi_{ 23}$&$\chi_{ 24}$&$\chi_{ 25}$&$\chi_{ 26} $
  \\\midrule
0&-2&0&0&0&0&0&0&0&0&0&0&0&0&0&0&0&0&0&0&0&0&0&0&0&0&0\\ 
1&0&0&1&1&0&0&0&0&0&0&0&0&0&0&0&0&0&0&0&0&0&0&0&0&0&0\\ 
2&0&0&0&0&1&1&0&0&0&0&0&0&0&0&0&0&0&0&0&0&0&0&0&0&0&0\\ 
3&0&0&0&0&0&0&0&0&0&1&1&0&0&0&0&0&0&0&0&0&0&0&0&0&0&0\\ 
4&0&0&0&0&0&0&0&0&0&0&0&0&0&0&0&0&0&0&0&2&0&0&0&0&0&0\\ 
5&0&0&0&0&0&0&0&0&0&0&0&0&0&0&0&0&0&0&0&0&0&0&0&0&2&0\\ 
6&0&0&0&0&0&0&0&0&0&0&0&0&0&0&0&0&0&0&0&0&0&2&0&0&0&2\\ 
7&0&0&0&0&0&0&0&0&0&0&0&0&0&0&0&0&0&2&2&0&0&0&2&2&2&2\\ 
8&0&0&0&0&0&0&0&0&0&0&0&1&1&0&1&1&2&0&0&2&2&2&4&2&2&6\\ 
9&0&0&0&0&0&0&0&0&2&2&2&0&0&2&2&2&0&2&2&2&4&4&4&8&8&10\\\bottomrule
                  \end{tabular}}
     \caption{\label{Mock_decompositions} \footnotesize{The first few Fourier coefficients of the terms $q^{-\frac{1}{8}+n}$ in the $q$-series $H_g(\t)$ and the corresponding representation $K_n$ (cf. Table \ref{Character Tables}).  
     }}
       \end{sidewaystable}

        \clearpage
 \begin{sidewaystable} 
\label{Character Tables}\centering
 \resizebox{1.0\textwidth}{!}{
 \begin{tabular}[H]{ccccccccccccccccccccccccccc}
 \toprule
classes&$1A$&$2A$&$2B$&$3A$&$3B$&$4A$&$4B$&$4C$&$5A$&$6A$&$6B$&$7A$&${7B}$&$8A$&$10A$&$11A$&$12A$&$12B$&$14A$&${14B}$&$15A$&${15B}$&$21A$&${21B}$&$23A$&${23B}$\\
\midrule
$\chi_1$&1 & 1 & 1 & 1 & 1 & 1 & 1 & 1 & 1 & 1 & 1 & 1 & 1 & 1 & 1 & 1 & 1 & 1 & 1 & 1 & 1 & 1 & 1 & 1 & 1 & 1\\
$\chi_2$&23&7&-1&5&-1&-1&3&-1&3&1&-1&2&2&1&-1&1&-1&-1&0&0&0&0&-1&-1&0&0\\
$\chi_3$&45 & - 3 & 5 & 0 & 3 & - 3 & 1 & 1 & 0 & 0 & - 1 & $e_7$ & $\bar e_7$& - 1 & 0 & 1 & 0 & 1 & - $e_7$ & - $\bar e_7$ & 0 & 0 & $e_7$ & $\bar e_7$& - 1 & - 1\\
$\chi_4$&${{45}}$&-3&5&0&3&-3&1&1&0&0&-1&$\bar e_ 7$&$e_ 7$&-1&0&1&0&1&-$\bar e_ 7$&-$e_ 7$&0&0&$\bar e_ 7$&$e_ 7$&-1&-1\\
$\chi_5$&231& 7& -9& -3& 0& -1& -1& 3& 1& 1& 0& 0& 0& -1& 1& 0& -1& 0& 0& 0& $e_{15}$& $\bar e_{15}$& 0& 0& 1& 1\\
$\chi_6$&${231}$&7&-9&-3&0&-1&-1&3&1&1&0&0&0&-1&1&0&-1&0&0&0&$\bar e_{15}$&$e_{15}$&0&0&1&1\\
$\chi_7$&252&28&12&9&0&4&4&0&2&1&0&0&0&0&2&-1&1&0&0&0&-1&-1&0&0&-1&-1\\
$\chi_8$&253&13&-11&10&1&-3&1&1&3&-2&1&1&1&-1&-1&0&0&1&-1&-1&0&0&1&1&0&0\\
$\chi_9$&483&35&3&6&0&3&3&3&-2&2&0&0&0&-1&-2&-1&0&0&0&0&1&1&0&0&0&0\\
$\chi_{10}$&770&-14&10&5&-7&2&-2&-2&0&1&1&0&0&0&0&0&-1&1&0&0&0&0&0&0&$e_{23}$&$\bar e_{23}$\\
$\chi_{11}$&${770}$&-14&10&5&-7&2&-2&-2&0&1&1&0&0&0&0&0&-1&1&0&0&0&0&0&0&$\bar e_{23}$&$e_{23}$\\
$\chi_{12}$&990&-18&-10&0&3&6&2&-2&0&0&-1&$e_ 7$&$\bar e_ 7$&0&0&0&0&1&$e_ 7$&$\bar e_ 7$&0&0&$e_ 7$&$\bar e_ 7$&1&1\\
$\chi_{13}$&${990}$&-18&-10&0&3&6&2&-2&0&0&-1&$\bar e_ 7$&$e_ 7$&0&0&0&0&1&$\bar e_ 7$&$e_ 7$&0&0&$\bar e_ 7$&$e_ 7$&1&1\\
$\chi_{14}$&1035&27&35&0&6&3&-1&3&0&0&2&-1&-1&1&0&1&0&0&-1&-1&0&0&-1&-1&0&0\\
$\chi_{15}$&1035&-21&-5&0&-3&3&3&-1&0&0&1&2$e_ 7$&2$\bar e_ 7$&-1&0&1&0&-1&0&0&0&0&-$e_ 7$&-$\bar e_ 7$&0&0\\
$\chi_{16}$&${1035}$&-21&-5&0&-3&3&3&-1&0&0&1&2$\bar e_ 7$&2$e_ 7$&-1&0&1&0&-1&0&0&0&0&-$\bar e_ 7$&-$e_ 7$&0&0\\
$\chi_{17}$&1265&49&-15&5&8&-7&1&-3&0&1&0&-2&-2&1&0&0&-1&0&0&0&0&0&1&1&0&0\\
$\chi_{18}$& 1771&-21&11&16&7&3&-5&-1&1&0&-1&0&0&-1&1&0&0&-1&0&0&1&1&0&0&0&0\\
$\chi_{19}$& 2024&8&24&-1&8&8&0&0&-1&-1&0&1&1&0&-1&0&-1&0&1&1&-1&-1&1&1&0&0\\
$\chi_{20}$& 2277&21&-19&0&6&-3&1&-3&-3&0&2&2&2&-1&1&0&0&0&0&0&0&0&-1&-1&0&0\\
$\chi_{21}$& 3312&48&16&0&-6&0&0&0&-3&0&-2&1&1&0&1&1&0&0&-1&-1&0&0&1&1&0&0\\
$\chi_{22}$& 3520&64&0&10&-8&0&0&0&0&-2&0&-1&-1&0&0&0&0&0&1&1&0&0&-1&-1&1&1\\
$\chi_{23}$& 5313&49&9&-15&0&1&-3&-3&3&1&0&0&0&-1&-1&0&1&0&0&0&0&0&0&0&0&0\\
$\chi_{24}$& 5544&-56&24&9&0&-8&0&0&-1&1&0&0&0&0&-1&0&1&0&0&0&-1&-1&0&0&1&1\\
$\chi_{25}$& 5796&-28&36&-9&0&-4&4&0&1&-1&0&0&0&0&1&-1&-1&0&0&0&1&1&0&0&0&0\\
$\chi_{26}$& 10395&-21&-45&0&0&3&-1&3&0&0&0&0&0&1&0&0&0&0&0&0&0&0&0&0&-1&-1\\
                   \bottomrule
  \end{tabular}}
  \caption{ \label{M24}\footnotesize{Character table of $M_{24}$. See \cite{atlas}. We adopt the naming system of \cite{atlas} and use the notation $e_n = \frac{1}{2}(-1+\sqrt{-n})$. }}
 \end{sidewaystable}
 
 \clearpage

\bibliography{RadM24.bib}{}

\end{document}